\documentclass[11pt]{article}
\usepackage{a4wide}
\usepackage{amsmath, amsthm}
\usepackage{amssymb}
\usepackage{graphicx,xcolor}
\usepackage{xfrac}
\usepackage{relsize}

\newcommand{\hidden}[1]{}

\theoremstyle{plain}
\newtheorem{thm}{Theorem}
\newtheorem{cor}{Corollary}
\newtheorem{lem}{Lemma}
\newtheorem{prop}{Proposition}

\makeatletter
\@addtoreset{lemA}{subsection}
\makeatother

\theoremstyle{definition}
\newtheorem{dfn}{Definition}

\theoremstyle{claim}

\theoremstyle{fact}

\theoremstyle{remark}
\newtheorem{rem}{Remark}

\theoremstyle{example}
\newtheorem{example}{Example}

\newtheoremstyle{case}{}{}{}{}{}{.}{ }{}
\theoremstyle{case}

\newcommand{\ca}[1]{\mathcal{#1}}

\newcommand{\ve}{\varepsilon}
\newcommand{\beq}{\begin{equation}}
\newcommand{\eq}{\end{equation}}
\newcommand{\supp}{\text{supp}}

\newcommand{\what}{\widehat}
\newcommand{\R}{\mathbb{R}}
\newcommand{\N}{\mathbb{N}}
\newcommand{\Z}{\mathbb{Z}}
\newcommand{\I}{\mathbb{I}}

\newcommand{\cA}{\mathcal{A}}

\newcommand{\cQ}{\mathcal{Q}}
\newcommand{\cS}{\mathcal{S}}

\newtheorem{thkh}{\rm {\bf Theorem \!\! KS}\!\!}

\begin{document}

\title{Inhomogeneous Diophantine Approximation on $M_0$-sets}

\author{Volodymyr Pavlenkov\footnote{Research is supported by the British Academy in partnership with the Academy of Medical Sciences, the
Royal Academy of Engineering, the Royal Society and Cara under the \emph{Researchers at Risk Fellowships} Award RAR\textbackslash 100132} \\ {\small\sc University of York}\\{\small\sc Igor Sikorsky Kyiv Polytechnic Institute }  \and ~Evgeniy Zorin \\ {\small\sc University of York }}

\date{}

\maketitle

\begin{abstract}
We prove new quantitative Schmidt-type theorem for Diophantine approximations with restraint denominators on fractals (more precisely, on $M_0$-sets). Our theorems introduce a sharp balance condition between the growth rate of the sequence of denominators and the decay rate of the Fourier transform of a Rajchman measure. Among the other things, this allows applications to sequences of denominators of polynomial growth. In particular, we infer new inhomogeneous Khintchine-J\"arnik type theorems with restraint denominators for a broad family of denominator sequences. Furthermore, our results provide non-trivial lower bounds for Hausdorff dimensions of intersections of two sets of inhomogeneously well-approximable numbers with restraint denominators.
\end{abstract}



\section{Introduction}
Metric theory of Diophantine approximations originates in the classical problem of approximating irrational numbers, such as $\sqrt{2}$ or $\pi$, by fractions. It was known already in ancient times that some constants arising in practice admit good rational approximations, providing nice precision while having a relatively not so large denominator (for example, rational approximations to $\pi$ such as $22/7\approx 3.14\dots$ and $355/113\approx 3.141593\dots$). A natural question to ask is to what extent such good rational approximations are typical for real numbers? In other terms, for a generic real number, what one should expect to have as a balance between the precision of a rational approximation and the size of its denominators? A bit surprisingly,
almost all rational numbers allow essentially the same asymptotic order of the best rational approximations when compared to the size of denominators of these approximations. This fact was established by Khintchine in~1929 (see Theorem~KS in Subsection~\ref{section_KS} below). Later, Schmidt refined this result in~1964 by proving that, for almost all numbers, the best rational approximations appear at quite regular intervals.

The development of metric theory of Diophantine approximations conducted to research on more complicated situations. For example, studies of simultaneous applications led to considering approximations of one real number by convergents (the best rational approximations) of another one, which naturally led to studies of approximations with constraints on available denominators. At the same time, in another direction of research, the metric theory of Diophantine approximations was extended to the context of fractal sets. Motivation for the latter was manifold, ranging from pure extension of existing theory to the demand by practical applications. Importantly, many of sets of interest arising in the metric theory of Diophantine approximations are fractals in nature (e.g. the sets of badly approximable numbers, well approximable numbers, the sets of exact order etc.), so the studies of Diophantine approximations on fractals provide, among the other things, a potent tool for investigating intersections of such sets.

So, a natural progress in this field has led 
to a large stream of research 
on metric theory of Diophantine approximations on fractals. It would have been an impossible task to give a comprehensive literature review on this subject within the framework of this paper, so we benefit from the possibility to refer the reader to~\cite{durham} for a general review of the subject and to~\cite{MIMO} for some examples of applications.

This paper 
develops further metric theory of Diophantine approximations on so-called $M_0$-sets.

\begin{dfn}\label{def1}
The set $F$ is called an {\em $M_0$-set} if it supports a non-atomic probability Borel measure $\mu$ ($\supp(\mu) \subset F$) whose Fourier transform $\widehat{\mu}$  vanishes at infinity, i.e.
$$\lim_{|t| \to \infty} |\widehat{\mu}(t)| = 0,$$
where
$$ \widehat{\mu}(t) \, := \int   e^{-2\pi it x}  \, \mathrm{d}\mu (x),      \hspace{1cm} (t\in \Z )\, .  \vspace{2mm}$$
Such a measure $\mu$ is called a Rajchman measure.
\end{dfn}

The interest of studying $M_0$-sets is bifold. On the one hand, there are quite many classical (fractal) sets of interest that support a Rajchman measure (see, for example, \cite{Bl,B,Hambrook,Kaufman,IFS,QR}). On another hand, it is quite natural to expect, at least heuristically, that such measures' behaviour have significant similarities with Lebesgue measure: indeed, one would naturally expect that, as $|\hat{\mu}(t)|$ tends to 0 as $t\to 0$, the behaviour of $\mu$ is determined to a significant degree (whatever it means in concrete terms) by some first coefficients, $\hat{\mu}(t)$ for $t=-N,\dots,N$. And then, of course, the measure with Fourier transform $\sum_{t=-N}^N\hat{\mu}(t)\exp\left(-2\pi i t\right)$ is absolutely continuous with respect to Lebesgue.
See a more detailed discussion and some further references in~\cite{povezazo}, where some quantitative results that illustrate the heuristics above have been proven and some other results from the past literature have been discussed from this point of view.


It is important to remark straight away that to have any meaningful generalization of Khintchine Theorem to the case of sufficiently general $\mu$, for instance to suitable Rajchman measures, we necessarily have to consider approximations with restrained sequences of denominators. Indeed, a classical construction by Kaufman~\cite{Kaufman}, later updated by Queffelec-Ramar\'e~\cite{QR}, shows that the set of badly approximable numbers supports a Rajchman measure (indeed, an infinite family of such measures). So, it is well possible that, with respect to certain Rajchman measures $\mu$, even with Fourier transform quickly converging to 0, $\mu$-almost all numbers are badly approximable, contrary to what we should have had with the usual Khintchine Theorem.

In~\cite{povezazo}, Khintchine type theorem (in fact, Schmidt-type counting theorem) is proven for sufficiently quickly growing sequences of denominators $\cA=(q_n)_{n\in\N}$. 
In fact, in~\cite{povezazo} the convergence part of Khintchine theorem is established in quite a general situation: for arbitrary sequence of denominators $\cA=(q_n)_{n\in\N}$ and 
under one of the following two conditions on the measure $\mu$
\begin{equation} \label{cond2_2}
\sum_{n=1}^{\infty} \; \max_{k\in\Z / \{ 0\} }|\widehat{\mu}(kq_n)|< \infty,
\end{equation}
\begin{equation} \label{cond2}
\sum_{n=1}^{\infty}   \, \sum_{k\in\Z / \{ 0\} }  \frac{|\widehat{\mu}(kq_n)|}{|k|}< \infty.
\end{equation}
However, the most challenging divergence part is established only for Rajchman measures with a sufficiently quickly decaying Fourier transform and, what is the most relevant for the topic of this paper, under quite an important growth condition on the sequence $\ca{A} = (q_n)_{n\in \N}$ (in particular, the growth of $\cA$ required there strictly exceeds the polynomial one, and indeed should be of order of magnitude $q_n\geq\exp\left(n^c\right)$ for some $c>0$).

In this paper, we prove a quantitative Schmidt-type theorem under a balance condition between the growth rate of the sequence $\ca{A} = (q_n)_{n\in \N}$ and the decay rate of the Fourier transform $\widehat{\mu}$. This balance condition allows, on the one hand, to establish Khintchine Theorem for some well-known measures (for instance, Kaufmann measures on the sets on badly- or well-approximable numbers) for the sequences of denominators of $\cA$ of polynomial growth. On another hand, 
this balance condition allows 
applications even to Rajchman measures with a relatively slowly decaying Fourier transform. For example, recently it was proven that, under quite broad conditions, there exist Rajchman measures on self-similar sets associated to contractions, but the decay rate of the Fourier transforms of the measures constructed is, when $t\to\infty$, $O(\log(|t|)^{-\delta})$, where $\delta>0$ is a constant not explicitly given in the paper, and the methods implied suggest that $\delta>0$ is very small (much smaller than one, for instance). In such situation, our main result, Theorem~\ref{mainTHM}, still allows to obtain a Khintchine type theorem for sufficiently sparse sequences of denominators.

The plan of this paper is as follows.

In Section~\ref{section_KS} we give some definitions and introduce some basic notation needed in our further discussion, and, in particular, to state our main result. In the same section we state Khintchine's theorem and also explain Schmidt's result, both of them are mentioned in the introduction.

In Section~\ref{Section_main_result} we state the main result of this paper.

In Section~\ref{Section_applications} we provide some examples of applications of our main result.

Finally, we give proofs in Section~\ref{Section_proof}.

\section{Khintchine-Sz\"{u}sz $0$ and $1$ Law and Schmidt generalization} \label{section_KS}

We begin by giving a few of basic definitions.

Given a real number $\gamma\in\I $, approximation function $\psi:\mathbb{N}\rightarrow \I$ and a natural number $q \in \N$,  let
\begin{equation} \label{def_En}
E(q, \gamma, \psi):=   \left\{ x\in \I : \| qx-\gamma \|\leq \psi(q) \right\},
\end{equation}
where $\| \alpha \|:=\min \{|\alpha-m|:m\in\Z\}$ denotes  the distance from $\alpha\in \R $ to the nearest integer.
For any sequence $\ca{A} = (q_n)_{n\in \N} \subset \N$ of natural numbers, define, for any $N\in\N$, the counting function by
\begin{equation}\label{countdef}
R(x,N) \, = \, R(x,N;\gamma,\psi, \ca{A}) := \# \big\{ 1\leq n \leq N : x \in E(q_n, \gamma, \psi)  \big\}.
\end{equation}
Note that the definition above could be rewritten as
\begin{equation}\label{countdef_chi}
R(x,N)=\sum_{n=1}^N \chi_{_{E(q_n, \gamma, \psi)}} (x),
\end{equation}
where $\chi_{_{E(q_n, \gamma, \psi)}}$ is the characteristic function of the set $E(q_n, \gamma, \psi).$

Recall that the set of inhomogeneous $\psi$-well approximable real numbers
\begin{equation}\label{wellset}
W_{\ca{A}}(\gamma; \psi):= \left\{ x\in \I: \|q_nx-\gamma \| \leq \psi(q_n) \text{ for infinitely many } n\in\mathbb{N} \right\}. \vspace{2mm}
\end{equation}
Note that
\begin{equation}\label{wellset_limsup}
W_{\ca{A}}(\gamma; \psi)=\limsup E(q_n, \gamma, \psi)=\bigcap_{n=1}^{\infty}\bigcup_{k=n}^{\infty}E(q_k, \gamma, \psi).
\end{equation}

We gathered in Proposition~\ref{prop1} below some basic properties of the function $R(x,N)$ and sets $W_{\ca{A}}(\gamma; \psi)$ that we will use in the proof of our main result.
\begin{prop}\label{prop1}
Let $x \in \I,$ $\gamma\in\I$ and let $\psi, \psi_1, \psi_2$ be auxiliary functions, that is $\psi, \psi_1, \psi_2:\mathbb{N}\rightarrow \I$. Let
$\ca{A} = (q_n)_{n\in \N} \subset \N$ be a sequence of natural numbers. Then,

\textbf{\emph{(i)}} $x \in W_{\ca{A}}(\gamma; \psi)$ if and only if $$\mathop{\lim}\limits_{N \to \infty}R(x,N) =\infty;$$

\textbf{\emph{(ii)}} if $\psi_1(q_n) \le \psi_2(q_n)$ for all $n \in \N,$ then
$$W_{\ca{A}}(\gamma; \psi_1) \subset W_{\ca{A}}(\gamma; \psi_2),$$
and
$$R(x,N;\gamma,\psi_1, \ca{A}) \le R(x,N;\gamma,\psi_2, \ca{A});$$

\textbf{\emph{(iii)}} if for some auxiliary functions $\psi_1, \psi_2$ we have $x \notin W_{\ca{A}}(\gamma; \psi_2)$, then
$$R(x,N;\gamma,\max\{\psi_1,\psi_2\}, \ca{A})=R(x,N;\gamma,\psi_1, \ca{A})+O(1).$$
\end{prop}

\begin{proof}[Proof of Proposition \ref{prop1}]
Statements (i) and (ii) of Proposition \ref{prop1} follow directly from definitions \eqref{def_En}, \eqref{countdef} and \eqref{wellset}.
In order to prove (iii), note that from (ii) of Proposition \ref{prop1} we have, for any $N \in \N$,
\begin{equation}\label{eq1}
  R(x,N;\gamma,\psi_1, \ca{A}) \le R(x,N;\gamma,\max\{\psi_1,\psi_2\}, \ca{A}) \le R(x,N;\gamma,\psi_1, \ca{A}) + R(x,N;\gamma,\psi_2, \ca{A}).
\end{equation}
Since $x \notin W_{\ca{A}}(\gamma; \psi_2),$ (i) of Proposition \ref{prop1} implies that
$R(x,N;\gamma,\psi_2, \ca{A})$ remains bounded as $N \to \infty.$ This fact together with \eqref{eq1} completes the proof of Proposition \ref{prop1}.
\end{proof}

Khintchine-Sz\"{u}sz Theorem provides the $0$ and $1$ law for the Lebesgue size of the set $W_{\N}(\gamma; \psi)$. Khintchine \cite{Khintchine24} proved the homogeneous statement (the case $\gamma = 0$) in 1924, and later, in 1954, Sz\"{u}sz \cite{szusz} generalized Khintchine's result to the inhomogeneous case.

\begin{thkh}
{\em Let $\gamma\in\I  $ and $\psi:\mathbb{N}\rightarrow \I$ be a real, positive  non-increasing function.  Then}
$$ m\big(W_{\N}(\gamma; \psi)  \big) =\left\{
\begin{array}{ll}
  0 & \mbox{if} \;\;\; \sum\limits_{n=1}^{\infty}\psi(n)  <\infty\;
      ,\\[2ex]
  1 & \mbox{if} \;\;\; \sum\limits_{n=1}^{\infty}\psi(n)
      =\infty \; ,
\end{array}\right .$$
{\em where $m$ is the Lebesgue measure. }
\end{thkh}

In 1964 Schmidt~\cite{schfracparts} generalized Khintchine-Sz\"{u}sz Theorem giving a quantitative result on the size of counting function $R(x,N)$ given by~\eqref{countdef} with decreasing auxiliary function $\psi$ and $\ca{A} = \N$: 

\[
 R(x,N)  \ = \   2\Psi(N)  +  O\Big(\Psi(N)^{1/2}\left(\log(\Psi(N))\right)^{2+\varepsilon}\Big), \quad N \in \N,
\]
for every $\varepsilon>0$ and for $m$-almost all $x,$
where
\begin{equation*}
\Psi(N):=\sum_{n=1}^N\psi(n).
\end{equation*}

Our main result, given in Section~\ref{Section_main_result} below, is a Schmidt-type counting theorem.


\section{Main result} \label{Section_main_result}
The statement of our main theorem uses the following definition.

\begin{dfn}\label{def2}
Let $\ca{A}= (q_n)_{n\in \N} $  be an increasing sequence of natural numbers and let $\alpha\in(0,1)$ be a real number. We say that $ \ca{A} $  is \emph{$\alpha$-separated} if there exists $m_0\in\N$ such that, for all $m,n \in \N$, $m_0\leq m<n$, the set of solutions $(s,t) \in \N^2$ of the following system of Diophantine inequalities
$$\begin{cases}
    1 \leq |sq_m-t q_n|< q_m^{\alpha}, \\
    s \le m^{5},
  \end{cases}$$
is empty.
\end{dfn}

\begin{rem}
Note that Definition~\ref{def2} is similar, but not identical, to the one in~\cite{povezazo} (see \cite[p.~12]{povezazo}). In Definition~\ref{def2} we have optimized one exponent, so we have $m^5$ in place of $m^{12}$ in~\cite{povezazo}. As a result, if a sequence of denominators $\cA$ is $\alpha$-separated in the sense of~\cite{povezazo}, than it is necessarily $\alpha$-separated in the sense of Definition~\ref{def2} (but not necessarily vice versa).
\end{rem}

\begin{thm} \label{mainTHM}
Let $\mu$ be a non-atomic probability  measure supported on a subset $F$ of
$ \ \I \, $.  Let $\ca{A}= (q_n)_{n\in \N} $ be an $\alpha$-separated increasing sequence of natural numbers for some $\alpha\in(0,1).$
Suppose there exists a real constant $\rho > 2$ 
and a monotonically decreasing function $h:\N\to\I$ verifying
\begin{equation}\label{growcond_f}
  h(q_n)\, = \, O\ (n^{-\rho}), \quad n \in \N.
\end{equation}
such that
\begin{equation}\label{growcond}
  |\widehat{\mu}(t)|\, = \, O\ (h(|t|)), \quad t \in \Z.
\end{equation}
Then, for all given $\gamma\in\I,  $ $\psi:\mathbb{N}\rightarrow \I$ and for
any $\varepsilon>0$ the counting function $R(x,N)$ satisfies
\begin{equation} \label{countFSPresultsv}
\begin{array}{ll}
 R(x,N)  \ = \   2\Psi(N)  + O\Big(\left(\Psi(N)+E(N)\right)^{1/2}\left(\log(\Psi(N)+E(N)+2\right)\big)^{2+\varepsilon}\Big)
\end{array}
\end{equation}
for $\mu$-almost all $x\in F$,  where
\begin{equation} \label{def_Psi}
\Psi(N):=\sum_{n=1}^N\psi(q_n) \,
\end{equation}
and
\begin{equation} \label{error2sv}
E(N) \ := \ \mathop{\sum\sum}_{1\leq m<n\leq N} (q_m,q_n)  \ \min\left(\frac{\psi(q_m)}{q_m},\frac{\psi(q_n)}{q_n}\right)  \,,
\end{equation}
here $(q_m,q_n)$ is the gcd of natural numbers $q_m$ and $q_n.$
\end{thm}

\begin{rem}\label{remassympt}
  If, under all conditions of Theorem~\ref{mainTHM}, for some $\varepsilon>0$ the gcd term $E(N),$ given by~\eqref{error2sv}, satisfies
  equality
  \begin{equation}\label{gcdbound}
    E(N)=O \left(\Psi^{2-\varepsilon}(N)\right), \quad N \in \N,
  \end{equation}
  and $\Psi(N) \to \infty$ as $N \to \infty,$ then for the counting function $R(x,N)$ we have the following asymptotic statement
  \begin{equation}\label{countasympt}
    \lim_{N \to \infty} \frac{R(x,N)}{2\Psi(N)}=1,
  \end{equation}
  for $\mu$-almost all $x\in F.$
\end{rem}

We postpone the proof of Theorem~\ref{mainTHM} until Section~\ref{Section_proof}. In the following Section~\ref{Section_applications} we give some examples of applications of this theorem.

\section{Applications of main result} \label{Section_applications}
In this section we 
give some examples of application of Theorem~\ref{mainTHM}.
\subsection{Example of sequence $\ca{A}$ with polynomial growth and fitting all conditions of Theorem~\ref{mainTHM}}



In this part, we construct, as an example, a family of sequences of denominators $\cA$ of polynomial growth for which Theorem~\ref{mainTHM} gives non-trivial results.
\begin{example} \label{example_PC}
In this example, we provide a construction of an $\alpha$-separated sequence $(q_m)_{m\in\mathbb{N}}$ of polynomial growth (more precisely, verifying~\eqref{example_PC_ub} for parameters $\rho_1<\rho_2$ as below).

Let
\begin{equation} \label{example_PC_assumption_rho_one_rho_two}
1<\rho_1<6\rho_1<\rho_2
\end{equation}
and $c>1$ be real parameters. 
Let us choose a sequence of integers $(n_k)_{k\in\N}$ as follows. First, choose $n_1$ to be a sufficiently large positive integer, so that $n_1^{\rho_1}<n_1^{\rho_2}/c$. Then, for every $k\in\N$, choose arbitrarily $n_{k+1}$ to be an integer in the range
\begin{equation} \label{choice_n_k_plus_1}
\lfloor  n_k^{\rho_2}\rfloor/c \leq n_{k+1} \leq \lfloor  n_k^{\rho_2}\rfloor.
\end{equation}
For each $k\in\N$, define sets of integers
\begin{equation} \label{def_cQk}
\cQ_k:=\left\{ s\cdot n_k \mid s=1,\dots, \lfloor n_k^{\rho_1-1} \rfloor \right\}.
\end{equation}
Then, define $\cA=(q_m)_{m\in\N}$ to be the set of numbers
\[
\bigcup_{k\in\N}\cQ_k
\]
put in increasing order.

We claim that we have, for every $m\in\N$ sufficiently large (so that we have $m>n_1^{\rho_1-1}$),
\begin{equation} \label{example_PC_ub}
\frac{\log q_m}{\log m}\leq \frac{\rho_2}{\rho_1-1}+1.
\end{equation}
Indeed, every $q_m$ has a form $s\cdot n_k$ for some $s,k\in\N$, where $s\leq n_k^{\rho_1-1}$. 
The lower bound on $m$ above means that $k\geq 2$. Then,
\[
m\geq s+\lfloor n_{k-1}^{\rho_1-1}\rfloor \geq s+\lfloor n_k^{(\rho_1-1)/\rho_2}\rfloor.
\]
Consider two cases: the first one when $s\leq \lfloor n_k^{(\rho_1-1)/\rho_2}\rfloor$, and the second one when $s>\lfloor n_k^{(\rho_1-1)/\rho_2}\rfloor$.

\noindent \textbf{Case~1.} Assume $s\leq \lfloor n_k^{(\rho_1-1)/\rho_2}\rfloor$. Then, we have $m\geq n_k^{(\rho_1-1)/\rho_2}$ (we use here $s\geq 1$), hence
\[
\frac{\log q_m}{\log m}\leq \frac{\log s + \log n_k}{\log n_k^{(\rho_1-1)/\rho_2}}\leq\frac{\log n_k^{(\rho_1-1)/\rho_2}+\log n_k}{\log n_k^{(\rho_1-1)/\rho_2}}\leq \frac{\rho_2}{\rho_1-1}+1.
\]
This proves~\eqref{example_PC_ub} in the first case.

\noindent \textbf{Case~2.} Assume $s > \lfloor n_k^{(\rho_1-1)/\rho_2}\rfloor$. As $s$ is an integer, we also have then $s >  n_k^{(\rho_1-1)/\rho_2}$, hence
\[
\frac{\log q_m}{\log m}\leq \frac{\log s + \log n_k}{\log s}\leq 1+ \frac{ \log n_k}{\log s} < \frac{\rho_2}{\rho_1-1}+1.
\]
So we have verified~\eqref{example_PC_ub} in both cases. 

Furthermore, let's remark for a further use
\begin{equation} \label{example_PC_lb}
q_m>m.
\end{equation}
Indeed, it follows from our construction that, for every $m\in\mathbb{N}$,
\[
q_{m+1}-q_m>1,
\]
and the claim follows by induction.

We proceed with proving that
\begin{equation} \label{example_PC_claim_alpha}
\text{ the sequence } (q_n) \text{ is } \alpha\text{-separated for } \alpha=1/\rho_1.
\end{equation}

So, let $m,n,s\in\mathbb{N}$ verifying $m<n$ and $s<m^5$. In case if, for some $k\in\mathbb{N}$, $q_m,q_n\in\mathcal{Q}_k$, then $q_m$ and $q_n$ are both divisible by $n_k$, so in case if, for some $t\in\mathbb{N},$ $|sq_m-tq_n|\geq 1$, we necessarily have
\[
|sq_m-tq_n|\geq n_k,
\]
hence
\[
|sq_m-tq_n|\geq q_m^{1/\rho_1}.
\]
To deal with the complimentary case, if $q_m\in\mathcal{Q}_k$, $q_n\in\mathcal{Q}_l$ for some $k<l$, note first that it could be deduced from the construction of the sequence $(q_m)_{m\in\mathbb{N}}$, that $q_m\geq m^{\rho_1}$. Then, we have
\[
|sq_m|<m^5 q_m < q_m^{6}\leq n_k^{6\rho_1}.
\]
So, for $n_k$ sufficiently large (otherwise speaking, for $m$ sufficiently large) and taking into account~\eqref{example_PC_assumption_rho_one_rho_two}, we have
\[
|sq_m|<\lfloor  n_k^{\rho_2}\rfloor/(2c)\leq n_{k+1}/2\leq q_n/2.
\]
Thus in this case we have that either $s=t=0$, in which case, of course, $sq_m-tq_n=0$, or
\[
|sq_m-tq_n|\geq q_m.
\]
This complets the verification of~\eqref{example_PC_claim_alpha}.
\end{example}

\begin{example} \label{example_gcd}
We continue to work in the framework of Example~\ref{example_PC}, now assuming in addition
$c\geq 2$  in~\eqref{choice_n_k_plus_1}.
We want to update the construction from that example to ensure, apart from polynomial growth and $\alpha$-separation, that also the term $E(N)$, defined in the statement of Theorem~\ref{mainTHM}, has a size that allows optimal result in Theorem~\ref{mainTHM}. More precisely, we are going to rectify Example~\ref{example_PC} above, by constructing a sequence $\widetilde{\cA}=(\widetilde{q}_t)_{t\in\N}$, so that we have
\begin{equation} \label{example_gcd_EN_result}
E(N)=\ \mathop{\sum\sum}_{1\leq m<n\leq N} (\widetilde{q}_m,\widetilde{q}_n)  \ \min\left(\frac{\psi(\widetilde{q}_m)}{\widetilde{q}_m},\frac{\psi(\widetilde{q}_n)}{\widetilde{q}_n}\right) =O\left(\sum_{n=1}^N\psi(\widetilde{q}_n)\right).
\end{equation}
Recall that we assume $c=2$. We choose all the numbers in the sequence $(n_k)_{k\in\N}$ to be prime, this is possible by Bertrand's postulate. 

Next, we modify sets $\cQ_k$, defined in~\eqref{def_cQk}, as follows:
\begin{equation} \label{def_cQk}
\widetilde{\cQ}_k:=\left\{ s\cdot n_k \mid s=1,\dots, \lfloor n_k^{\rho_1-1} \rfloor, \; s \text{ is prime} \rfloor \right\}.
\end{equation}
Then, similarly to Example~\ref{example_PC}, we define $\widetilde{\cA}=(\widetilde{q}_t)_{t\in\N}$ to be the set of numbers
\[
\bigcup_{k\in\N}\widetilde{\cQ}_k
\]
put in increasing order.

For the sake of comparison, let us denote by $\cA$ the set of denominators built as in Example~\ref{example_PC} using the same sequence of primes $(n_k)_{k\in\N}$ that we have just constructed for $\widetilde{\cA}$. Then, it follows from the definitions that the set of values of $\widetilde{\cA}$ is the subset of values of $\cA$, that is, for every index $t\in\N$ there exists an index $m\in\N$ such that $\widetilde{q}_t=q_m$. It follows then from the law of distribution of prime numbers that, for $t\in\N$ large enough, $t\asymp\frac{m}{\log m}$, hence, for every $\varepsilon>0$, we have, for all indices $t\in\N$ large enough,
\[
\frac{\log\widetilde{q}_t}{\log t}\leq \frac{\rho_2}{\rho_1-1}+1+\varepsilon,
\]
so $\left(\widetilde{q}_t\right)_{t\in\N}$ has a polynomial growth.

We proceed with establishing~\eqref{example_gcd_EN_result}. We claim that, for every $n\in\N$, we have
\begin{equation} \label{sum_gcd_O}
\sum_{m=1}^{n-1}\frac{(\widetilde{q}_m,\widetilde{q}_n)}{\widetilde{q}_n}=O(1),
\end{equation}
which clearly implies~\eqref{example_gcd_EN_result}.

By construction, $\widetilde{q}_n=s\cdot n_k$ for some $k\in\N$ and prime $s$ verifying $s\leq n_k^{\rho_1-1}$. Similarly, for all $m\in\N$,  $\widetilde{q}_m=s_1\cdot n_l$. Let us denote by $I_1$ the collection of indices $m<n$ such that $l=k$, and $I_2$ the complimentary set of indices $m<n$ such that $l<k$. Naturally, $\{1,\dots,n-1\}=I_1\sqcup I_2$, hence
\[
\sum_{m=1}^n\frac{(\widetilde{q}_m,\widetilde{q}_n)}{\widetilde{q}_n}=\sum_{m\in I_1}\frac{(\widetilde{q}_m,\widetilde{q}_n)}{\widetilde{q}_n}+\sum_{m\in I_2}\frac{(\widetilde{q}_m,\widetilde{q}_n)}{\widetilde{q}_n}.
\]
We are going to show that both sums in the right-hand side are $O(1)$, which is equivalent to proving~\eqref{sum_gcd_O}.

First, we consider the sum $\sum_{m\in I_1}$. In this case, by definition of $I_1$, $k=l$. Then, we necessarily have $s_1< s$, because $m<n$. Moreover, we have in this case
\[
(q_m,q_n)=(s_1 n_k, s n_k)=n_k,
\]
because $s_1$ and $s$ are two distinct primes. So,
\[
\sum_{m\in I_1}\frac{(\widetilde{q}_m,\widetilde{q}_n)}{\widetilde{q}_n}=\sum_{\substack{2\leq s_1<s \\ s_1 \text{ prime}
}}\frac{1}{s}<1.
\]

Next, we consider the sum $\sum_{m\in I_2}$. In this case, we have $l<k$, hence, by construction of sequences $(n_t)_{t\in\N}$ and $\widetilde{\cA}$, we necessarily have $\widetilde{q}_m\leq n_{k-1}^{\rho_1}\leq n_k^{\rho_1/\rho_2}$, hence $\widetilde{q}_n\geq n_k>\widetilde{q}_m^{\rho_2/\rho_1}$. We find thus
\[
\sum_{m\in I_2}\frac{(\widetilde{q}_m,\widetilde{q}_n)}{\widetilde{q}_n}\leq \sum_m\frac{\widetilde{q}_m}{\widetilde{q}_m^{\rho_2/\rho_1}}<+\infty,
\]
where at the last step we use~\eqref{example_PC_assumption_rho_one_rho_two}. This proves~\eqref{sum_gcd_O}, hence~\eqref{example_gcd_EN_result}.

\end{example}

\subsection{Khintchine Theorem on the set of Liouville numbers}

In this section, we establish the Khintchine-Sz\"usz theorem (for restraint denominators) on the set of Liouville numbers:
\[
\mathbb{L}=\left\{x \in \mathbb{R} \setminus \mathbb{Q} \;\; |  \; \; \forall \; n \in \mathbb{N} \;\; \exists \;\; q \in \mathbb{N}: \| q x \|<q^{-n} \right\}.
\]

In 2000 Bluhm~\cite{blum} constructed a Rajchman measure $\mu_{\mathbb{L}}$ supported on $\mathbb{L}$ showing that the set of Liouville numbers is an $M_0$ set. In 2002 Bugeaud~\cite{Bugeaund} calculated the exact decay rate of the Fourier transform $\widehat{\mu_{\mathbb{L}}}:$

\begin{equation}\label{decay_Liouville}
  |\widehat{\mu_{\mathbb{L}}}(t)| \le \exp \left\{-c_2 \sqrt{\log (1+ |t|)}\right\}, \quad t \in \Z,
\end{equation}
where $c_2$ is a positive absolute constant appearing in~\cite{Bugeaund}. We will use constant $c_2$ and measure $\widehat{\mu_{\mathbb{L}}}$ in the statement of Theorem~\ref{LiouvilleTHM} below).

This allows us to infer the following theorem.

\begin{thm} \label{LiouvilleTHM}
Let $\psi:\mathbb{N}\rightarrow \I$, $\gamma\in\I$, $\alpha\in(0,1)$ and let $\ca{A}= (q_n)_{n\in \N} $ be an $\alpha$-separated increasing sequence of natural numbers with growth rate
\begin{equation}\label{grow_rate_for_Liouv}
  \log q_n \ge \left(\frac{\rho}{c_2} \log n \right)^2, \quad n \ge n_0,
\end{equation}
for some natural $n_0$, $\rho > 2$ and $c_2$ from~\eqref{decay_Liouville}. Recall the measure $\widehat{\mu_{\mathbb{L}}}$ described before the statement of the theorem and appearing in~\eqref{decay_Liouville}.

Then, for any $\varepsilon>0$, we have the following asymptotic counting result for $\widehat{\mu_{\mathbb{L}}}:$-almost every Liouville number $x$:
\begin{equation*}
\begin{array}{ll}
 R(x,N)  \ = \   2\Psi(N)  + O\Big(\left(\Psi(N)+E(N)\right)^{1/2}\left(\log(\Psi(N)+E(N)+2\right)\big)^{2+\varepsilon}\Big),
\end{array}
\end{equation*}
where
\begin{equation*}
\Psi(N):=\sum_{n=1}^N\psi(q_n) \,
\end{equation*}
and
\begin{equation*}
E(N) \ := \ \mathop{\sum\sum}_{1\leq m<n\leq N} (q_m,q_n)  \ \min\left(\frac{\psi(q_m)}{q_m},\frac{\psi(q_n)}{q_n}\right).
\end{equation*}
\end{thm}
\begin{proof}[Proof of Theorem~\ref{LiouvilleTHM}]
We will deduce Theorem~\ref{LiouvilleTHM} from Theorem~\ref{mainTHM}. In order to do this we put
$$h(t)=\exp \left\{-c_2 \sqrt{\log (1+ t)}\right\}, \quad t \in \N.$$
Conditions \eqref{growcond_f} and \eqref{growcond} of Theorem~\ref{mainTHM} now follow from~\eqref{decay_Liouville} and~\eqref{grow_rate_for_Liouv}, therefore all statements of Theorem~\ref{LiouvilleTHM} now follow from the corresponding statements of Theorem~\ref{mainTHM}.
\end{proof}

Consider the following two examples of sequences $\cA$. We will use these examples in Corollary~\ref{cor_Liouville} below.

\begin{enumerate}
\item 
Let $\cS$
be a finite set of $k$ distinct primes $p_1,\ldots, p_k$. 
It is shown in~\cite{povezazo} that any sequece of $\cS$-smooth numbers $\cA_1$ is an $\alpha$-separated sequence, for any $\alpha \in (0,1)$. It is also shown there that $\cA_1$ has a growth rate
\[
\log q_n > \frac{\log 2}{2} n^{\frac{1}{k}}, \quad n \ge 2,
\]
and has the gcd error term $E(N)$ of order $E(N)=O(\Psi(N))$ (see Theorem~5 in~\cite{povezazo}).

\item 
Let $\ca{A}_2$ be a subsequence of sequence $\widetilde{\ca{A}}$ verifying
the growth condition \eqref{grow_rate_for_Liouv}. 
Since $\widetilde{\ca{A}}$ is an $\alpha$-separated sequence and verifies
\eqref{example_gcd_EN_result}, necessarily its subsequence $\ca{A}_2$ is also $\alpha$-separated and verifies \eqref{example_gcd_EN_result}.
\end{enumerate}

The following corollary is readily implied by Theorem~\ref{LiouvilleTHM} and properties of sequences $\ca{A}_1$ and $\ca{A}_2.$

\begin{cor}\label{cor_Liouville}
Let $\mathbb{L}$ be the set of Liouville numbers with Rajchman measure $\mu_{\mathbb{L}}$ supported on it and verifying~\eqref{decay_Liouville}. 
Then, for any $\gamma \in \I$ and constructed above sequences $\ca{A}_1$ and $\ca{A}_2$, we have
  $$ \mu_{\mathbb{L}}\big(W_{\ca{A}_i}(\gamma; \psi) \cap \mathbb{L} \big) =\left\{
\begin{array}{ll}
  0 & \mbox{if} \;\;\; \sum\limits_{n=1}^{\infty}\psi(q_n)  <\infty\;
      ,\\[2ex]
  1 & \mbox{if} \;\;\; \sum\limits_{n=1}^{\infty}\psi(q_n)
      =\infty \; , \quad \quad i=1,2.
\end{array}\right.$$
\end{cor}

\subsection{Hausdorff and Fourier dimensions of sets $W_{\ca{A}}(\gamma; \psi)$ and their intersections }

In this section we deduce a nontrivial result on the Hausdorff and Fourier dimensions of some subsets of the sets $W_{\ca{A}}(\gamma; \psi)$.

Recall that the Fourier dimension of a set $A \subset \R$ is a number $\dim_F A,$ that is a supremum
of all $\eta \in [0,1]$ such that
\begin{equation}\label{dimfurdef}
  |\widehat{\mu}(t)| = O\left(|t|^{-\frac{\eta}{2}}\right), \quad \text{as} \quad |t| \to \infty,
\end{equation}
for some Borel non-atomic probability measure $\mu$ on $\R$ with $\supp(\mu) \subset A.$ More precisely, the supremum is
taken over all Borel non-atomic probability measures, supported on A and satisfying \eqref{dimfurdef}.

It is known (see, for example,~\cite{Mattila1,Mattila2}) that for any Borel set $A \subset \R$
\begin{equation}\label{hasfurineq}
  \dim_F A \le \dim_H A.
\end{equation}
A set $A \subset \R,$ for which $\dim_F A = \dim_H A,$ is called a Salem set.

For the set of inhomogeneous $\psi$-well approximable real numbers $W_{\ca{A}}(\gamma; \psi)$ an upper bound of the Hausdorff
dimension is known (see, for example,~\cite{Hambrook}):
\begin{equation}\label{hausdupperbound}
  \dim_H W_{\ca{A}}(\gamma; \psi) \le \min\left\{\tau(\ca{A}, \psi), 1\right\},
\end{equation}
where
\begin{equation}\label{const_in_upper_h_bound}
  \tau(\ca{A}, \psi)= \inf \left\{\eta \ge 0 \,\, : \,\, \sum_{q \in \ca{A}} q \left(\frac{\psi(q)}{q}\right)^{\eta}<\infty\right\}.
\end{equation}

The next theorem gives a nontrivial result about the Fourier dimension of the set of inhomogeneous $\psi$-well approximable real numbers lying inside
some other set $A$ and follows from Theorem~\ref{mainTHM}.
\begin{thm}\label{THM_four_dim_inters}
Let $A \subset \R$ be a Borel set and let $\dim_F A = d \ne 0$. Let $\alpha \in (0,1)$, and let $\ca{A}= (q_n)_{n\in \N} $  be an increasing $\alpha$-separated sequence of natural numbers.
Assume that, for some $\delta>0$ and $n_0 \in \N$, the sequence $\ca{A}$ has the growth rate
\begin{equation}\label{gr_rate_dim_inters}
  q_n \ge n^{\frac{4}{d}+\delta}, \quad n \ge n_0.
\end{equation}
Let $\psi:\mathbb{N}\rightarrow \I$ be an approximation function.
Assume that
\begin{equation} \label{diverg_cond}
\Psi(N):=\sum_{n=1}^N\psi(q_n) \to \infty \quad \text{as} \quad N \to \infty.
\end{equation}
Then, for any $\gamma \in \I$,
$$\dim_F \left(W_{\ca{A}}(\gamma; \psi) \cap A \right) =d.$$
\end{thm}
\begin{proof}[Proof of Theorem~\ref{THM_four_dim_inters}]
Firstly note that the definition of the Fourier dimension implies that for any $A_1, A_2 \subset \R$
$$\dim_F \left(A_1 \cap A_2\right) \le \min \left\{\dim_F A_1, \dim_F A_2\right\}.$$
Therefore,
$$\dim_F \left(W_{\ca{A}}(\gamma; \psi) \cap A \right) \le \dim_F A=d.$$
So, in order to prove Theorem~\ref{THM_four_dim_inters}, we need to show that
\begin{equation}\label{fou_dim_ge_d}
  \dim_F \left(W_{\ca{A}}(\gamma; \psi) \cap A \right) \ge d,
\end{equation}
which means that the intersection $W_{\ca{A}}(\gamma; \psi) \cap A$ supports some Borel non-atomic probability measure
with a suitable decay rate of its Fourier transform.

Let us fix a positive number $\widetilde{\varepsilon} < \frac{d^2 \delta}{8+2d\delta}.$
  Since $\dim_F A = d \ne 0,$ then, for any $\varepsilon \in (0, \tilde{\varepsilon})$,
  there exists a non-atomic probability measure $\mu_A^{(\varepsilon)},$ supported on $A$ and satisfying
  \begin{equation}\label{dimfurdef_thm}
  |\widehat{\mu_A^{(\varepsilon)}}(t)| = O\left(|t|^{-\frac{d}{2} + \varepsilon}\right), \quad \text{as} \quad |t| \to \infty.
\end{equation}
 Put
  $$h(t)=t^{-\frac{d}{2} + \widetilde{\varepsilon}}, \quad t \in \N.$$
  Now, for any $\varepsilon \in (0, \tilde{\varepsilon})$ we can apply Theorem~\ref{mainTHM}, choosing $\mu=\mu_A^{(\varepsilon)},$ $F=A$ and $h$ as above. Conditions \eqref{growcond_f} and \eqref{growcond} of Theorem~\ref{mainTHM}
  follow from~\eqref{dimfurdef_thm} and \eqref{gr_rate_dim_inters} respectively. By Theorem~\ref{mainTHM}, we get that, for any $\varepsilon \in (0, \tilde{\varepsilon})$ and for $\mu_A^{(\varepsilon)}$-almost all $x \in A$,
  equality~\eqref{countFSPresultsv} holds true. Taking into account~\eqref{diverg_cond}, from~\eqref{countFSPresultsv} we get that
  $$R(x, N) \to \infty \; \text{ as } N \to \infty, \quad \text{for} \,\, \mu_A^{(\varepsilon)} \text{-almost all} \,\, x \in A,$$
  so for $\mu_A^{(\varepsilon)}$-almost all $x \in A$ we have $x \in W_{\ca{A}}(\gamma; \psi).$
  So, for any $\varepsilon \in (0, \tilde{\varepsilon})$, 
we have
  $$\mu_A^{(\varepsilon)}\left(W_{\ca{A}}(\gamma; \psi) \cap A \right)=1.$$
  For any $\varepsilon \in (0, \tilde{\varepsilon})$ consider a measure space $(\R, \mathfrak{B}(\R), \mu_A^{(\varepsilon)}).$
  Put $B=W_{\ca{A}}(\gamma; \psi) \cap A.$ Now we modify measure $\mu_A^{(\varepsilon)}$ introducing a new Borel measure
  $$\mu_B^{(\varepsilon)}(C) := \mu_A^{(\varepsilon)} (C \cap B), \quad \forall C \in \mathfrak{B}(\R).$$
  Observing the full measure statement $\mu_A^{(\varepsilon)}(B)=1$ and the fact that $\mu_A^{(\varepsilon)}$ is non-atomic and probability
  we deduce that new measure space $(\R, \mathfrak{B}(\R), \mu_B^{(\varepsilon)})$ is a probability space with a non-atomic measure $\mu_B^{(\varepsilon)}.$ From the definition of measures $\mu_B^{(\varepsilon)}$  it follows that
  $$\supp(\mu_B^{(\varepsilon)}) \subset B, \quad \forall\varepsilon \in (0, \tilde{\varepsilon}).$$
  Moreover, for any $t \in \Z$ and $\varepsilon \in (0, \tilde{\varepsilon})$
  $$\int_{\R}   e^{-2\pi it x}  \, \mathrm{d}\mu_A^{(\varepsilon)} (x) = \int_{B} + \int_{\R \setminus B}=
  \int_{B}   e^{-2\pi it x}  \, \mathrm{d}\mu_A^{(\varepsilon)} (x) = $$
  $$= \int_{B}   e^{-2\pi it x}  \, \mathrm{d}\mu_B^{(\varepsilon)} (x)= \int_{B} + \int_{\R \setminus B}=
  \int_{\R}   e^{-2\pi it x}  \, \mathrm{d}\mu_B^{(\varepsilon)} (x),$$
  so $\mu_B^{(\varepsilon)}$ has the same decay rate of Fourier transform, as $\mu_A^{(\varepsilon)}.$
  Therefore, observing \eqref{dimfurdef_thm} with $\mu_B^{(\varepsilon)}$ instead of $\mu_A^{(\varepsilon)},$
  $$\dim_F \left(W_{\ca{A}}(\gamma; \psi) \cap A \right) \ge d,$$
  so \eqref{fou_dim_ge_d} holds. This finishes the proof of Theorem~\ref{THM_four_dim_inters}.
\end{proof}

The following two corollaries immediately follows from Theorem~\ref{THM_four_dim_inters}
and inequalities
\begin{equation}\label{ineq_HFdimen_inters}
 \dim_F \left(A \cap W \right) \le \dim_H \left(A \cap W \right) \le \min \left\{\dim_H A, \dim_H W\right\} \le \dim_H A.
\end{equation}
In the above inequalities~\eqref{ineq_HFdimen_inters} sets $A$ and $W$ are two Borel subsets of $\R,$ these inequalities follow from~\eqref{hasfurineq} and
definitions of Hausdorff and Fourier dimensions.

\begin{cor}\label{cor_inter_is_salem}
  If, under all conditions of Theorem~\ref{THM_four_dim_inters}, we also have that
  $A$ is a Salem set, then the intersection $W_{\ca{A}}(\gamma; \psi) \cap A$
  is also a Salem set and
  $$\dim_H \left(W_{\ca{A}}(\gamma; \psi) \cap A\right) = \dim_F \left(W_{\ca{A}}(\gamma; \psi) \cap A\right)=d.$$
  \end{cor}

\begin{rem}\label{rem_salem}
In view of Corollary~\ref{cor_inter_is_salem}, for any given $d \in (0,1]$ we can build Salem subsets $W$ of the set of inhomogeneous $\psi$-well approximable real numbers $W_{\ca{A}}(\gamma; \psi)$ with $\dim_H W=d$.
  For example, in \cite{Hambrook} Hambrook proved that if the denominators sequence $\ca{A}_1$
  satisfies
  $$\sum_{q \in \ca{A}_1} \frac{1}{q}= \infty,$$
  and for the approximation function $\psi_1$ we have
  \begin{equation}\label{order_of_approx_func}
    \liminf_{M \to \infty} \frac{- \log \psi_1 (M)}{\log M} = \limsup_{M \to \infty} \frac{- \log \psi_1 (M)}{\log M} = \lambda,
  \end{equation}
  then for any $\gamma_1 \in \I$ the set $W_1=W_{\ca{A}_1}(\gamma_1; \psi_1)$ is Salem with
\begin{equation} \label{dimW_1}
\dim_H W_1= \dim_F W_1 = \min\{\frac{2}{1+\lambda}, 1\}.
\end{equation}
  So, the subset $W=W_{\ca{A}}(\gamma; \psi) \cap W_{\ca{A}_1}(\gamma_1; \psi_1),$ whose elements are both $\psi$ and $\psi_1$-well approximable (with different denominators sequences), is a Salem set. Of course, this readily extends to finite intersections of the sets of well-approximale numbers.

\end{rem}

\begin{cor}\label{cor_int_not_salem}
Let $\alpha \in (0,1)$, and let $\ca{A}= (q_n)_{n\in \N} $ be an increasing $\alpha$-separated sequence of natural numbers with a growth rate
\begin{equation} \label{q_n_gr_base}
q_n\geq n^{4+\delta}
\end{equation}
for some $\delta>0$. Assume that approximation function $\psi$ verifies~\eqref{diverg_cond}. Then,
\[
\dim_H W_{\ca{A}}(\gamma; \psi) = 1.
\]
\end{cor}
\begin{proof}
It is well-known that for every $d\in(0,1]$ there exists a Salem set with Fourier dimension $d$. For example, one could use the set $W_1$ mentioned in~\eqref{dimW_1} with $\lambda=2/d-1$. Then, the result follows from Theorem~\ref{THM_four_dim_inters} with $A=W_1$ (and the use of inequalities~\eqref{ineq_HFdimen_inters}).
\end{proof}

\begin{rem}\label{dim_is_1}
  Corollary~\ref{cor_int_not_salem} provides an inhomogeneous Khintchine theorem for a wide range of sequences of restraint denominators (more precisely, it provides the divergence case of Khintchine's theorem, but the convergence case always readily follows from classical Borel-Cantelli lemma).
  Homogeneous version of this result was previously established by Rynne~\cite{Rynne}. He shows that
  $$\dim_H W_{\ca{A}}(0; \psi)= \min\left\{\tau(\ca{A}, \psi), 1\right\},$$
  where $\tau(\ca{A}, \psi)$ is defined by~\eqref{const_in_upper_h_bound}. Note that~\eqref{diverg_cond} implies $\tau(\ca{A}, \psi) \ge 1,$
  so in homogeneous case, under~\eqref{diverg_cond}, we have $\dim_H W_{\ca{A}}(0; \psi)=1.$ Inhomogeneous case is not so well studied. The existing results cover $\ca{A}= \N$ and, more generally, not too sparse subsets of $\N$ (verifying the condition $\sum_{q \in \ca{A}_1} \frac{1}{q}= \infty$). Corollary~\ref{cor_int_not_salem}, complimentarily,  gives inhomogeneous results for sufficiently sparse denominators sequence (verifying~\eqref{gr_rate_dim_inters}). 

\end{rem}

Of course, it would be very desirable to extend the results of Theorem~\ref{THM_four_dim_inters} beyond the divergence condition~\eqref{diverg_cond}. This could be done, at least partially, by using a powerful Mass Transference Principle. Theorem~\ref{MTP} below states a particular case that we use of a general result \cite{BerVelMTP}[Theorem~3].
\begin{thm} \label{MTP}
Let $X\subset\R$ and let $\lambda,\mu>0$. Let $\left(B_i\right)_{i\in\N}$ be a sequence of balls in $X$ (considered as a metric space with the evident distance $d(x,y)=|x-y|$) such that radii of $B_i$ tend to 0 as $i\to\infty$. For each ball $B_i=B_i(x_i,r_i)$ define $B_i^{\lambda/\mu}:=B\left(x_i,r_i^{\lambda/\mu}\right)$.

Assume that, for every ball $B\subset X$,
\[
\dim_H\left(B\cap\limsup_{i\to\infty}B_i^{\lambda/\mu}\right)\geq\mu.
\]
Then, for every ball $B\subset X$,
\[
\dim_H\left(B\cap\limsup_{i\to\infty}B_i\right)\geq\min\left(\lambda,\dim_H X\right).
\]
\end{thm}

Theorem~\ref{THM_four_dim_inters_MT} below could be considered as a partial extension of Theorem~\ref{THM_four_dim_inters}. It provides a non-trivial lower bound for Hausdorff dimension of intersection of the set of well-approximable numbers with other sets.
\begin{thm}\label{THM_four_dim_inters_MT}
Let $A \subset \R$ be a Borel set and let $\dim_F A = d \ne 0$. Let $\alpha \in (0,1)$ and let $\ca{A}= (q_n)_{n\in \N} $  be an increasing $\alpha$-separated sequence of natural numbers. Assume that, for some $\delta>0$ and $n_0 \in \N$, the sequence $\ca{A}$ has the growth rate~\eqref{q_n_gr_base}.
Let $\psi:\mathbb{N}\rightarrow \I$ be an approximation function.
Then, for any $\gamma \in \I$,
\[
\dim_H \left(W_{\ca{A}}(\gamma; \psi) \cap A \right)\geq \tau(\ca{A},\psi)\cdot d,
\]
where $\tau(\ca{A},\psi)$ is defined by~\eqref{const_in_upper_h_bound}.
\end{thm}
\begin{proof}
The result follows from Theorem~\ref{THM_four_dim_inters}, inequalities~\eqref{ineq_HFdimen_inters} and Theorem~\ref{MTP} with $\lambda=d\cdot \tau(\ca{A},\psi)$ and $\mu=d$.
\end{proof}
\begin{example}
Let $\lambda\geq 0$, $\gamma_1\in\I$ and let $\psi_1$ and $\ca{A}_1$ be the same as in Remark~\ref{rem_salem} (for the sake of even further concretness, one could take $\psi_1: t\to t^{-\lambda}$ and $\ca{A}_1=\Z$). As it is discussed in Remark~\ref{rem_salem}, in this case $W_{\ca{A}_1}(\gamma_1,\psi_1)$ is a Salem set verifying
\[
\dim_H W_{\ca{A}_1}(\gamma_1,\psi_1)=\dim_F W_{\ca{A}_1}(\gamma_1,\psi_1)=\frac{2}{1+\lambda}.
\]
Further, let $a\in\N$, $a\geq 2$ and let $\ca{A}=\left(a^n\right)_{n\in\N}$. It is not hard to verify that the sequence $\ca{A}$ is $1/2$-separated. Then, it follows from Theorem~\ref{THM_four_dim_inters_MT} that, for any approximating function $\psi$,
\[
\dim_H\left(W_{\ca{A}_1}(\gamma_1,\psi_1)\cap W_{\ca{A}}(\gamma,\psi)\right)\geq \frac{2\cdot\tau(\ca{A},\psi)}{1+\lambda}.
\]
\end{example}

The following corollary extends Corollary~\ref{cor_int_not_salem}, providing thus a J\"arnik-type theorem for a wide range on sequences of restraint denominators.
\begin{cor}\label{THM_four_dim_inters_MT_2}
Let $\ca{A}= (q_n)_{n\in \N} $  be an increasing $\alpha$-separated sequence of natural numbers, for some $\alpha\in(0,1)$, and assume that, for some $\delta>0$ and $n_0 \in \N$, the sequence $\ca{A}$ has the growth rate~\eqref{q_n_gr_base}.
Let $\psi:\mathbb{N}\rightarrow \I$ be an approximation function.
Then, for any $\gamma \in \I$,
\[
\dim_H \left(W_{\ca{A}}(\gamma; \psi) \right)=\tau(\ca{A},\psi),
\]
where $\tau(\ca{A},\psi)$ is defined by~\eqref{const_in_upper_h_bound}.
\end{cor}
\begin{proof}
Because of~\eqref{hausdupperbound}, we need to prove only the lower bound
\[
\dim_H W_{\ca{A}}(\gamma; \psi)\geq \tau(\ca{A},\psi).
\]
But this result follows from Theorem~\ref{THM_four_dim_inters_MT} with the remark that the interval $A=[0,1]$ has Fourier dimension 1.
\end{proof}

\section{Establishing the main result} \label{Section_proof}

The proof of Theorem \ref{mainTHM} is based on the following lemma ~\cite[Lemma~1.5]{harman}.

\begin{lem} \label{ebc}
Let $(X,\ca{B},\mu)$ be a probability space, let $(f_n(x))_{n \in \N}$ be a sequence of non-negative $\mu$-measurable functions defined on $X$, and $(f_n)_{n \in \N },\ (\phi_n)_{n  \in \N}$ be sequences of real numbers  such that
$$ 0\leq f_n \leq \phi_n \hspace{7mm} (n=1,2,\ldots).$$

\noindent 
Suppose that for arbitrary  $a,b \in \N$ with $a <  b$, we have
\begin{equation} \label{ebc_condition1}
\int_{X} \left(\sum_{n=a}^{b} \big( f_n(x) -  f_n \big) \right)^2\mathrm{d}\mu(x)\, \leq\,  C\!\sum_{n=a}^{b}\phi_n
\end{equation}

\noindent for an absolute constant $C>0$. Then, for any given $\varepsilon>0$,  we have
\begin{equation} \label{ebc_conclusion}
\sum_{n=1}^N f_n(x)\, =\, \sum_{n=1}^{N}f_n\, +\, O\left(\Phi(N)^{1/2}\log^{\frac{3}{2}+\varepsilon}\Phi(N)+\max_{1\leq k\leq N}f_k\right)
\end{equation}
\noindent for $\mu$-almost all $x\in X$, where $
\Phi(N):= \sum\limits_{n=1}^{N}\phi_n
$. 
\end{lem}

A mechanism of applying Lemma \ref{ebc} for proving similar to our counting results can be found in ~\cite{povezazo}. We will
move in parallel to their proof taking into account that, unlike~\cite{povezazo}, we do not have a growth conditions on the
sequence $\ca{A}= (q_n)_{n\in \N}.$ and decay rate for $\hat{\mu}(t)$, which are replaced in our paper by balance condition~\eqref{growcond}.

\subsection{Choosing $f_n(x)$ and $f_n$ from Lemma \ref{ebc}}

Let us consider Lemma \ref{ebc}  with
\begin{equation} \label{Harman_choice_parameters}
X := \I \, , \qquad f_n(x):= \chi_{_{E(q_n, \gamma, \psi)}}(x)  \qquad {\rm and } \qquad
f_n=2\psi(q_n) \, ,
\end{equation}

It follows from~\eqref{countdef_chi} that, with this choice of $f_n(x)$,
$$%
\text{the l.h.s.  of  } \eqref{ebc_conclusion}  \ =   \ R(x, N).
$$

Note that $f_n(x)^2=f_n(x), \quad x \in \I,$, so, for any $a,b \in \N$ with $a <  b$,

\begin{equation*}
\begin{aligned}
\left(\sum_{n=a}^{b}(f_n(x)-f_n) \right)^2  &=  \ \left(\sum_{n=a}^{b}f_n(x)\right)^2 \, + \, \left(\sum_{n=a}^{b}f_n\right)^2 -  \ 2\sum_{n=a}^{b}f_n(x) \cdot \sum_{n=a}^{b}f_n \\[2ex]
   =  \  \sum_{n=a}^{b}f_n(x)  \  &+  \ 2\mathop{\sum\sum}_{a\leq m < n\leq b}f_m(x)f_n(x) \  +  \  \left(\sum_{n=a}^{b}f_n\right)^2  \ -  \ 2\sum_{n=a}^{b}f_n\cdot\sum_{n=a}^{b}f_n(x) \,  ,  \vspace{2mm}
\end{aligned}
\end{equation*}

\vspace{2mm}

\noindent and so it follows that
\vspace{2mm}
\begin{equation} \label{integral_square_estimate}
\begin{aligned}
  \text{the l.h.s.  of  } \eqref{ebc_condition1} 
  \ =  \   \sum_{n=a}^{b}\mu(E_n)  \, + \, 2 &\mathop{\sum\sum}_{a\leq m<n\leq b}  \mu(E_m\cap E_n)-  \\[2ex] \, &- \ 4 \sum_{n=a}^{b} \psi(q_n) \left( \sum_{n=a}^{b}\mu(E_n)  - \sum_{n=a}^{b} \psi(q_n) \right),
\end{aligned}
\end{equation}
here we used the following short notation $E_n:=E(q_n, \gamma, \psi), \quad n \in \N.$

\vspace{2mm}
Relation \eqref{integral_square_estimate} shows that we need to obtain a `good' estimates of the measure of sets $E_n=E(q_n, \gamma, \psi)$
and the measure of their intersections to satisfy the condition \eqref{ebc_condition1}.

\subsection{Estimating the measure of the sets $E_n$ and their intersections}
In this section we present different estimates that help us to prove the main result. We will use the Fourier
analysis estimating the measure of sets $E_n$
and their intersections.

Let $\varepsilon$ and $\delta$ be real numbers such that  $0<\varepsilon\leq 1$ and $0< \delta<1/4$.  Let  $ \chi_{\delta} : \I \to \R $  be the characteristic function defined by
$$ 
\chi_{\delta}(x):= \begin{cases}1 \ \text{ if }  \ \|x\|\leq \delta,  \\[1ex]
 0 \ \text{ if } \  \|x\|>\delta \, ,  \end{cases}    $$
and let  $\chi_{\delta, \varepsilon}^+  : \I \to  \R    $ and $  \chi_{\delta, \varepsilon}^-: \I \to \R$ be the continuous upper and lower approximations of  $ \chi_{\delta} $ given by  \vspace{2mm}
$$ 
\chi_{\delta, \varepsilon}^+(x):= \begin{cases} 1 & \text{\  if   }  \  \|x\|\leq \delta, \\[1ex]
1+ \dfrac{1}{\delta\varepsilon}(\delta-\|x\|) &\text{\   if }  \  \delta < \|x\| \leq (1+\varepsilon)\delta \\[1.3ex]
0 & \text{\ if } \  \|x\|>(1+\varepsilon)\delta \, , \end{cases}  $$
and

$$ 
 \chi_{\delta, \varepsilon}^-(x): = \begin{cases} 1 & \text{\  if } \  \|x\|\leq (1 -\varepsilon)\delta  \\[1ex]
\dfrac{1}{\delta\varepsilon}(\delta-\|x\|) &\text{\   if } \  (1-\varepsilon)\delta < \|x\| \leq \delta \\[1.3ex]
0  & \text{\ if } \ \|x\|>\delta  \, . \end{cases}
$$

\noindent Clearly, both  $\chi_{\delta, \varepsilon}^+  $ and $  \chi_{\delta, \varepsilon}^- $ are periodic functions with period 1. Next, given a real positive function $\psi:\mathbb{N}\rightarrow \I$ and any integer $q\geq 4$, consider the functions  $W_{q,\gamma,\varepsilon,\psi}^{+} $ and $W_{q,\gamma,\varepsilon,\psi}^{-}$ defined by \vspace{2mm}
\begin{equation} \label{Wdef}
W_{q,\gamma,\varepsilon}^{+}(x)=W_{q,\gamma,\varepsilon,\psi}^{+}(x)\ :=\ \Big( \sum_{p=0}^{q-1}\delta_{\frac{p+\gamma}{q}} (x)  \Big) *\chi_{\frac{\psi(q)}{q},\varepsilon}^{+}(x)
\end{equation}
and
\begin{equation*} \label{W-def}
W_{q,\gamma,\varepsilon}^{-}(x) =W_{q,\gamma,\varepsilon,\psi}^{-} (x) \  :=  \ \Big(\sum_{p=0}^{q-1}\delta_{\frac{p+\gamma}{q}} (x) \Big) *\chi_{\frac{\psi(q)}{q},\varepsilon}^{-} (x)  \, , \vspace{2mm}
\end{equation*}
where $    *  $   denotes convolution and $\delta_x$ denotes the Dirac delta-function at  the point $x\in\mathbb{R}$.

It occurs that
\begin{eqnarray*}
W_{q,\gamma,\varepsilon}^{+}(x) =  \sum_{p=0}^{q-1}\chi_{\frac{\psi(q)}{q}, \varepsilon}^{+}\left(\textstyle{x- \frac{p+\gamma}{q}}\right)
\end{eqnarray*}
and
\begin{eqnarray*}
W_{q,\gamma,\varepsilon}^{-}(x) = \sum_{p=0}^{q-1}\chi_{\frac{\psi(q)}{q}, \varepsilon}^{-}\left(\textstyle{x- \frac{p+\gamma}{q}}\right) \,  .
\end{eqnarray*}
It thus  follows that for any $0<\varepsilon \le 1$ and any integer $q\geq 4$, \vspace{2mm}
\begin{equation} \label{muineq}
\int_{0}^{1}W_{q,\gamma,\varepsilon}^-(x)\mathrm{d}\mu(x)
\; \leq  \;
\mu( E(q, \gamma, \psi)) \;  \leq \; \int_{0}^{1}W_{q,\gamma,\varepsilon}^+(x)\mathrm{d}\mu(x). \vspace{2mm}
\end{equation}
Now we need to consider the Fourier series expansions
$$ \sum_{k \in \Z }\widehat{W}_{q,\gamma,\ve}^{\pm}(k)  \exp(2\pi kix)
$$
 of $ W_{q,\gamma,\varepsilon}^{+}  $ and  $W_{q,\gamma,\varepsilon}^{-} $.

The values and basic estimates of the Fourier coefficients $\widehat{W}_{q,\gamma,\ve}^{\pm}(k)$ are presented in ~\cite{povezazo} (see Lemma 1 in ~\cite{povezazo}). For the convenience of the reader, we collect them in the following proposition.

\begin{prop}\label{prop2}
Let $0< \varepsilon,\tilde{\varepsilon} \le 1$ and $\widehat{W}_{q,\gamma,\ve}^{\pm}(k), \quad k \in \Z,$ are the Fourier coefficients of $ W_{q,\gamma,\varepsilon}^{+}  $ and  $W_{q,\gamma,\varepsilon}^{-} $. Then,  for any integers $q, r \geq 4:$

\textbf{\emph{(i)}} for $k\neq 0$,
\vspace*{2mm}
\begin{equation} \label{fcoef}
\widehat{W}_{q,\gamma,\varepsilon}^+(k)\, =\,\begin{cases}
\exp\left( -\dfrac{2\pi  i k \gamma}{ q}\right) \dfrac{q\left(\cos(2\pi k\psi(q)q^{-1})-\cos(2\pi k\psi(q)q^{-1}(1+\varepsilon))\right)}{2\pi^2 k^2\psi(q)q^{-1}\varepsilon}   &\text{\  if   }  \ q\mid k \\[3ex]
0 & \text{\  if   }  \  q\nmid k \,  ,
\end{cases}  \,
 \end{equation}
and for $k = 0$,  \vspace*{2mm}
\begin{equation} \label{fcoef_zero}
\widehat{W}_{q,\gamma,\varepsilon}^{+}(0)=(2+\varepsilon) \,  \psi(q)\, ;  \vspace{2mm}
\end{equation}

\textbf{\emph{(ii)}} for $k\neq 0$,
\vspace{2mm}
\begin{equation} \label{fcoef2}
\widehat{W}_{q,\gamma,\varepsilon}^-(k)\, =\,
\begin{cases}
\exp\left( -\dfrac{2\pi  i k \gamma}{ q}\right) \dfrac{q\left(\cos(2\pi k\psi(q)q^{-1}(1-\varepsilon))-\cos(2\pi k\psi(q)q^{-1}) \right)}{2\pi^2 k^2\psi(q)q^{-1}\varepsilon}  &\text{\  if   }  \ q\mid k \\[3ex]
0 & \text{\  if   }  \  q\nmid k  \, ,   \vspace{2mm}
\end{cases}
\end{equation}
and for $k=0$,
\vspace{2mm}
\begin{equation} \label{fcoef_zero2}
\widehat{W}_{q,\gamma,\varepsilon}^-(0)  = (2 - \ve)\, \psi(q)  \, ; \vspace{2mm}
\end{equation}

\textbf{\emph{(iii)}} for any $ s\in\mathbb{Z}   \setminus\{0\} $
\begin{eqnarray}
\left|\widehat{W}_{q,\gamma,\varepsilon}^{\pm}(s q)\right| &\leq& (2+\varepsilon) \, \psi(q),\label{W_ub_psi}\\[2ex]
\left|\widehat{W}_{q,\gamma,\varepsilon}^{\pm}(s q)\right| &\leq& \frac{1}{\pi^2s^2\psi(q)\varepsilon}.\label{W_ub_1s2}
\end{eqnarray}

\textbf{\emph{(iv)}}
\begin{eqnarray} \label{ub_single_sum}
\sum_{s \in \Z} \big|\what{W}_{q,\gamma,\varepsilon}^{\pm}(sq)\big| \  <  \ \frac{3}{\varepsilon^{1/2}}
\end{eqnarray}
and
\begin{equation} \label{ub_double_sum}
\sum_{s\in \Z} \,  \sum_{t\in \Z}  \, \big|\what{W}_{q,\gamma,\varepsilon}^{\pm}(sq)\big| \ \big|\what{W}_{r,\gamma,\tilde{\varepsilon}}^{\pm}(tr)\big|  \ \leq \  \frac{9}{\varepsilon^{1/2}\cdot\tilde{\varepsilon}^{1/2}}  \, .
\end{equation}
\end{prop}

By~\eqref{ub_single_sum}, $\sum\limits_{k \in \Z} \big| \widehat{W}_{q,\gamma,\varepsilon}^{\pm}(k) \big|  < \infty $, so the Fourier series
$$ \sum_{k \in \Z }\widehat{W}_{q,\gamma,\ve}^{\pm}(k)  \exp(2\pi kix)
$$
converges uniformly to $W_{q,\gamma,\varepsilon}^{\pm}(x)$ for all $ x \in \I$.  Hence, it  follows that
\begin{eqnarray*}
\int_0^1 W_{q,\gamma,\varepsilon}^{\pm}(x)  \; \mathrm{d}\mu(x) \; =  \;  \sum_{k\in \Z} \, \widehat{W}_{q,\gamma,\varepsilon}^{\pm}(k)  \;  \widehat{\mu}(-k)   \, .
\end{eqnarray*}
This together with \eqref{muineq}, \eqref{fcoef_zero},  \eqref{fcoef_zero2} and the fact that $\widehat{\mu}(0)=1$, implies that
\vspace*{3ex}
\begin{equation} \label{mu_ie}
\begin{array}{ll}
   \mu( E(q, \gamma, \psi))  \  \leq \ (2+\varepsilon) \, \psi(q)  \  + \displaystyle{\sum_{k \in \Z \setminus \{0 \} }}\widehat{W}_{q,\gamma,\varepsilon}^{+}(k)  \; \widehat{\mu}(-k)\;
   \\[6ex]
   \mu( E(q, \gamma, \psi))  \ \geq  \  (2-\varepsilon) \,  \psi(q)   \ + \displaystyle{\sum_{k \in \Z \setminus \{0 \} }}  \widehat{W}_{q,\gamma,\varepsilon}^{-}(k) \; \widehat{\mu}(-k)  \, .
    \end{array} \,
\end{equation}

\subsubsection{Estimating the sums from measures of $E_n$}
Now we are ready to prove the following estimates for the sums of $\mu(E_n).$

\begin{lem}\label{lemsizeEn}
Under the conditions of Theorem \ref{mainTHM}, 
we have, for arbitrary  $a,b \in \N$ with $a <  b$,
\begin{equation} \label{lem_sum_mu_good_result}
\sum_{n=a}^b\mu(E(q_n, \gamma, \psi)) = 2\sum_{n=a}^b\psi(q_n) \, + \, O\left( \min\Big(1,\sum_{n=a}^b\psi(q_n)\Big) \right).
\end{equation}
\end{lem}

We need the following well known statement in order to prove Lemma \ref{lemsizeEn}. We'll leave
this statement without citation because, in our opinion, it belongs to the mathematical folklore.

\begin{lem}\label{lemconvserie}
Let $(a_n, n \in \N)$ be a sequence of nonnegative real numbers verifying $a_1>0$ and $$S_n=a_1+a_2+ \ldots + a_n, \quad n \in \N.$$
Then for any real $\xi >0$ the series
$$\sum_{n=1}^{\infty}\frac{a_n}{S_n^{1+\xi}}$$
converges.
\end{lem}
\begin{proof}[Proof of Lemma \ref{lemconvserie}]
Firstly, note that, for any $n \in \N$,
\begin{equation}\label{intforlem}
  \int_{S_{n-1}}^{S_n} \frac{dx}{x^{1+\xi}} \ge \int_{S_{n-1}}^{S_n} \frac{dx}{S_n^{1+\xi}}=\frac{a_n}{S_n^{1+\xi}}.
\end{equation}
To prove the convergency of the series
$$\sum_{n=1}^{\infty}\frac{a_n}{S_n^{1+\xi}}$$
consider it partial sum
$$B_N=\sum_{n=1}^{N}\frac{a_n}{S_n^{1+\xi}}.$$
In view of \eqref{intforlem}, for any $N \in \N,$
$$B_N=\frac{a_1}{S_1^{1+\xi}}+ \frac{a_2}{S_2^{1+\xi}} +\frac{a_3}{S_3^{1+\xi}} + \ldots +\frac{a_N}{S_N^{1+\xi}} \le$$
$$\le \frac{a_1}{S_1^{1+\xi}} + \int_{S_1}^{S_2} \frac{dx}{x^{1+\xi}} +\int_{S_2}^{S_3} \frac{dx}{x^{1+\xi}} +\ldots +\int_{S_{N-1}}^{S_N} \frac{dx}{x^{1+\xi}}= \frac{1}{a_1^{\xi}} + \int_{S_1}^{S_N} \frac{dx}{x^{1+\xi}}=$$
$$=\frac{1}{a_1^{\xi}} + \frac{1}{\xi} \left(\frac{1}{S_1^{\xi}}-\frac{1}{S_N^{\xi}}\right) \le \frac{1}{a_1^{\xi}}\left(1+\frac{1}{\xi}\right).$$
So, the sequence of partial sums $(B_N, N\in \N)$ for the series with nonnegative elements is bounded, which implies the convergency of this series.
\end{proof}

Now we can prove Lemma \ref{lemsizeEn}.

\begin{proof}[Proof of Lemma \ref{lemsizeEn}] We prove Lemma \ref{lemsizeEn} under a weaker condition on number $\rho$ from \eqref{growcond_f}, compared to Theorem \ref{mainTHM}, namely $\rho > \frac{3}{2}$.
For any given sequence of real numbers $(\varepsilon_n)_{n=a}^b$ in $(0,1]$, it follows from ~\eqref{mu_ie} that
\begin{equation} \label{eq2}
\left|\mu(E_n)-2\psi(q_n)\right| \leq \psi(q_n)\varepsilon_n+
\max_{\circ\in\{+,-\}}\left|\sum_{k\in \Z \setminus \{0\}}\widehat{W}_{q_n,\gamma,\varepsilon_n}^{\circ}(k)\widehat{\mu}(-k)\right|, \quad a \le n \le b.
\end{equation}

It follows from~\eqref{fcoef} and  \eqref{fcoef2} that $\widehat{W}_{q_n,\gamma,\varepsilon_n}^{\pm}(k)=0$
unless $k = s q_n$ for some integer $s.$ Also from \eqref{growcond} we have, that $ |\what{\mu}(sq_n)| \ll n^{-\rho}, $ where $\rho>\frac{3}{2}.$ So, using \eqref{ub_single_sum},

\begin{equation}\label{eq3}
 \left|\sum_{k\in \Z \setminus \{0\}}\widehat{W}_{q_n,\gamma,\varepsilon_n}^{\pm}(k)\widehat{\mu}(-k)\right| \le \sum_{s \in \Z\setminus \{0\}} \big|\what{W}_{q,\gamma,\varepsilon}^{\pm}(sq)\big| \big|\what{\mu}(sq_n)\big|\ll \frac{3}{n^{\rho}\varepsilon_n^{\frac{1}{2}}}.
\end{equation}

Now, in view of \eqref{eq2} and \eqref{eq3},
\begin{equation}\label{lem_sum_mu_good_ie2_2}
\left|\sum_{n=a}^b\mu(E_n)-2\sum_{n=a}^b\psi(q_n)\right|\ll \sum_{n=a}^b\psi(q_n)\varepsilon_n+\sum_{n=a}^b\frac{3}{n^{\rho}\varepsilon_n^{\frac{1}{2}}},
\end{equation}
therefore, for $\varepsilon_n=1,$ $a \leq n \leq b$,
\begin{equation} \label{lem_sum_mu_good_result_half1}
\left|\sum_{n=a}^b\mu(E_n)-2\sum_{n=a}^b\psi(q_n)\right|\ll\sum_{n=a}^b\psi(q_n) \, .
\end{equation}

To complete the proof we need to get one more upper bound, namely
\begin{equation} \label{lem_sum_mu_good_result_half2}
\left|\sum_{n=a}^b\mu(E_n)-2\sum_{n=a}^b\psi(q_n)\right|\ll 1.
\end{equation}
We can deduce the upper bound \eqref{lem_sum_mu_good_result_half2} again from \eqref{lem_sum_mu_good_ie2_2} choosing suitable values of $(\varepsilon_n)_{n=a}^b.$ Since $\rho>\frac{3}{2},$ let's choose a real number $\xi \in (0, \, 2\rho -3).$ For any $n\in\N$ let
$$
\Psi(n):=\sum_{k=1}^n\psi(q_k)  \quad {\rm and } \quad \varepsilon_n:=\min\left(1,\Psi(n)^{-1-\xi}\right).
$$
By definition,  $|\psi(q_n)|\leq 1$  and so
$
\varepsilon_n^{-1}\leq n^{1+\xi} \, .
$
So, for the second term on the r.h.s. of \eqref{lem_sum_mu_good_ie2_2} we have
\begin{equation} \label{lem_sum_mu_good_second_sum_bounded}
\sum_{n=a}^b\frac{3}{n^{\rho}\varepsilon_n^{1/2}} \, \leq  \, \sum_{n=1}^\infty\frac{3}{n^{\rho - \frac{\xi}{2}-\frac{1}{2}}} \, <  \, \infty  \,,
\end{equation}
since $\rho - \frac{\xi}{2}-\frac{1}{2}>1$ because of the choice of $\xi$.


It follows from Lemma~\ref{lemconvserie} that, for $\xi>0$, the series
$$\sum_{n=1}^{\infty}\frac{\psi(q_n)}{\Psi(n)^{1+\xi}}$$
converges, so
$$\sum_{n=1}^{\infty}\frac{\psi(q_n)}{\max\left(1,\Psi(n)^{1+\xi}\right)}$$
converges also. Therefore,
\begin{equation} \label{lem_sum_mu_good_ie3}
\sum_{n=a}^b\psi(q_n)\varepsilon_n \, < \, \sum_{n=1}^{\infty}\frac{\psi(q_n)}{\max\left(1,\Psi(n)^{1+\xi}\right)} \, < \, \infty \, .
\end{equation}
The upper bound \eqref{lem_sum_mu_good_result_half2} now follows from the inequalities ~\eqref{lem_sum_mu_good_ie2_2}, \eqref{lem_sum_mu_good_second_sum_bounded} and~\eqref{lem_sum_mu_good_ie3}.
\end{proof}

\subsubsection{Estimating the sums from measures of intersections $E_n \cap E_m$}

Estimating the sums from measures of intersections of the sets $E_n$ is technically a more complicated work. Here we use
the condition on $\alpha$-separability (see Definition \ref{def2}), which we don't use proving Lemma \ref{lemsizeEn}.

Before formulating the statement on the upper bound for the sums from measures of intersections we prove the following
result, related to the case $\xi =0$ of Lemma~\ref{lemconvserie}.
\begin{lem}\label{lemdivserie}
Let $(a_n, n \in \N)$ be a sequence of nonnegative real numbers, $a_1 \ne 0$ and $$S_n=a_1+a_2+ \ldots + a_n, \quad n \in \N.$$
Then for any $N \in \N$ for the partial sums $B_N$ of the series
$$\sum_{n=1}^{\infty}\frac{a_n}{S_n}$$
we have
$$B_N=\sum_{n=1}^{N}\frac{a_n}{S_n} \le 1+ \log(S_N) - \log(a_1).$$
\end{lem}
\begin{proof}[Proof of Lemma \ref{lemdivserie}]
Firstly note that for any $n \in \N$
\begin{equation}\label{intforlem1}
  \int_{S_{n-1}}^{S_n} \frac{dx}{x} \ge \int_{S_{n-1}}^{S_n} \frac{dx}{S_n}=\frac{a_n}{S_n}.
\end{equation}
In view of \eqref{intforlem1}, for any $N \in \N,$
$$B_N=\frac{a_1}{S_1}+ \frac{a_2}{S_2} +\frac{a_3}{S_3} + \ldots +\frac{a_N}{S_N} \le$$
$$\le \frac{a_1}{S_1} + \int_{S_1}^{S_2} \frac{dx}{x} +\int_{S_2}^{S_3} \frac{dx}{x} +\ldots +\int_{S_{N-1}}^{S_N} \frac{dx}{x}= 1 + \int_{S_1}^{S_N} \frac{dx}{x}=$$
$$= 1+ \log(S_N) - \log(S_1) = 1+ \log(S_N) - \log(a_1).$$
\end{proof}

Now we can formulate and prove the statement on the upper bound for the sums from measures of intersections.
\begin{lem}\label{lemsizeEnEm}
Let for any $\tau >1$
\begin{equation}\label{estimpsiqn}
  \psi(q_n) \ge n^{-\tau}, \quad n \in \N.
\end{equation}
 Then, under all conditions of Theorem \ref{mainTHM}, for arbitrary  $a,b \in \N$ with $a <  b$,
  \begin{multline}\label{overlapfinal_finite_set_of_primesSV}
2 \mathop{\sum\sum}_{a\leq m<n\leq b} \mu(E_m\cap E_n) \ \leq \  \left( \sum_{n=a}^{b}   \mu(E_n)  \right)^2 + O\left(\mathop{\sum\sum}_{a\leq m<n\leq b} (q_m,q_n)\min\left(\frac{\psi(q_m)}{q_m},\frac{\psi(q_n)}{q_n}\right)\right) \\[2ex]
+ O\left(\left(\sum_{n=a}^{b}\psi(q_n)\right)\log^+\left(\sum_{n=a}^{b}\psi(q_n)\right)+\sum_{n=a}^{b}\psi(q_n)\right),
\end{multline}
where $\log^+(x):=\max(0,\log(x)), \quad x>0.$
\end{lem}
\begin{proof}[Proof of Lemma \ref{lemsizeEnEm}]
In order to receive suitable estimates of the measures of intersections we starts with some Fourier analysis (once again, we move here in parallel with
the corresponding reasonings in~\cite{povezazo}). So, let's fix a sequence of real numbers $(\varepsilon_n)_{n\in \N} \subset (0,1]$ and consider functions
\begin{equation} \label{wmndef}
W_{m,n}^{+}(x):= W_{q_m,\gamma,\varepsilon_m}^+(x) \cdot W_{q_n,\gamma,\varepsilon_n}^+(x), \quad x \in \R, \, m,n \in \N, \, m<n,
\end{equation}
(recall that function $W_{q,\gamma,\varepsilon}^{+}$ is defined by \eqref{Wdef}).
Then,
\begin{eqnarray*}
\mu(E_m\cap E_n) &\leq& \int_{0}^{1} W_{q_m,\gamma,\varepsilon_m}^+(x) W_{q_n,\gamma,\varepsilon_n}^+(x) \, \mathrm{d}\mu(x)   \nonumber
 \\ [2ex]
&=& \int_{0}^{1}W_{m,n}^{+}(x)\mathrm{d}\mu(x).
\end{eqnarray*}
The Fourier coefficients of the product $W_{m,n}^{+}$ are convolutions
\begin{eqnarray}  \label{greek}
\what{W}_{m,n}^{+}(k) & := &   \, \int_{0}^{1}
W_{q_m,\gamma,\varepsilon_m}^{+}(x)W_{q_n,\gamma,\varepsilon_n}^{+}(x) \, \exp(-2\pi kix) \mathrm{d}x    \nonumber \\[2ex]
 & = &  \sum_{j \in \Z} \what{W}_{q_m,\gamma,\varepsilon_m}^{+}(j) \
\what{W}_{q_n,\gamma,\varepsilon_n}^{+}(k-j), \quad k \in \Z.
\end{eqnarray}
Moreover, $\sum\limits_{k \in \Z} \big| \widehat{W}_{m,n}^{+}(k) \big|  < \infty $, so the Fourier series
$$ \sum_{k \in \Z }\widehat{W}_{m,n}^{+}(k)  \exp(2\pi kix)
$$
converges uniformly to $W_{m,n}^{+}(x)$ for all $ x \in \I$.
Hence, it  follows that
\begin{eqnarray*}
\int_0^1 W_{m,n}^{+}(x)  \; \mathrm{d}\mu(x) \; =  \;  \sum_{k\in \Z} \, \widehat{W}_{m,n}^{+}(k)  \;  \widehat{\mu}(-k)   \, .
\end{eqnarray*}
So,
\begin{eqnarray}\label{intersbound}
\mu(E_m\cap E_n) = \mu (E(q_m, \gamma, \psi) \cap E(q_n, \gamma, \psi))&\leq&
\\ [2ex] \leq \sum_{k \in \Z }\what{W}_{m,n}^{+}(k)\what{\mu}(-k) \nonumber
&=& \what{W}_{m,n}^{+}(0) + \sum_{k\in \Z \setminus \{0\}} \what{W}_{m,n}^{+}(k)\what{\mu}(-k).  \label{mit1}  \vspace{2mm}
\end{eqnarray}

We consider the two terms on the right hand side  of \eqref{intersbound} separately. From \eqref{greek} with $k=0,$ (and refer to \eqref{def_En}) we have
\begin{equation} \label{ub_W0_lacunary}
\begin{aligned}
\what{W}_{m,n}^{+}(0) &:=   \, \int_{0}^{1}W_{q_m,\gamma,\varepsilon_m}^{+}(x)W_{q_n,\gamma,\varepsilon_n}^{+}(x)\mathrm{d}x \leq \\[2ex]
&\leq  \, \Big|E\left(q_m, \gamma, (1+\varepsilon_m)\cdot \psi\right) \cap E\left(q_m, \gamma, (1+\varepsilon_n)\cdot \psi\right)\Big|\, ,  \vspace{2mm}
\end{aligned}
\end{equation}
where $ | \, . \, | $ is Lebesgue measure. It is relatively straightforward to verify (see~\cite[Equation~3.2.5]{harman}\footnote{Equation~3.2.5 in \cite{harman} as stated is not correct   -- the `big O' error term is missing.} for the details) that for   any  $q, q' \in \N$
\begin{equation*}
| E(q, \gamma, \psi) \cap E(q', \gamma, \psi)| \, =  \, 4\psi(q)\psi(q')+ O\left( (q,q') \ \min\Big(\frac{\psi(q)}{q},\frac{\psi(q')}{q'}\Big)  \right) \, .
\end{equation*}
Hence, it follows that
\begin{multline*} \label{main_q_Harman_ub}
\Big|E\left(q_m, \gamma, (1+\varepsilon_m)\cdot \psi\right) \cap E\left(q_m, \gamma, (1+\varepsilon_n)\cdot \psi\right)\Big| \,  = \, 4(1+\varepsilon_m)(1+\varepsilon_n)\psi(q_m)\psi(q_n)    \\[1ex]   + \ \  O\left( (q_m,q_n)\min\Big(\frac{\psi(q_m)}{q_m},\frac{\psi(q_n)}{q_n}\Big)  \right) \, .
\end{multline*}
This together with~\eqref{ub_W0_lacunary}  implies that
\begin{equation}  \label{main_q_Harman_ub2}
 \what{W}_{m,n}^{+}(0)  \, \leq  \,  4(1+\varepsilon_m)(1+\varepsilon_n)\psi(q_m)\psi(q_n)  +  O\left( (q_m,q_n)\min\Big(\frac{\psi(q_m)}{q_m},\frac{\psi(q_n)}{q_n}\Big)  \right) \, ,
\end{equation}
which give us the estimate of the first term on the right hand side  of \eqref{mit1}.

We proceed with considering the second term, which we will denote by $S_{m,n}$. In view of \eqref{fcoef} and \eqref{greek}, it follows that
\begin{eqnarray} \label{def_S_m_n}
S_{m,n} &:= & \sum_{k \in\mathbb{Z}\setminus \{0\} } \what{W}_{m,n}^{+}(k)\what{\mu}(-k) = \nonumber \\[2ex]
&= & \mathop{ \mathop{\sum\sum}_{ s,t\in\mathbb{Z} } }_{ sq_m-tq_n \neq 0}
\what{W}_{q_m,\gamma,\varepsilon_m}^{+}(sq_m)
\what{W}_{q_n,\gamma,\varepsilon_n}^{+}(tq_n)
\what{\mu}\left(-(sq_m+tq_n)\right).
\end{eqnarray}
We decompose $S_{m,n}$ into three sums:
$$ S_{m,n}= S_1(m,n)+S_2(m,n)+S_3(m,n), \vspace{2mm} $$
where \vspace{2mm}
\begin{eqnarray*}
S_1(m,n) &:=&  \sum_{t\in \Z \setminus \{0\}}   \what{W}_{q_m,\gamma,\varepsilon_m}^+(0) \what{W}_{q_n,\gamma,\varepsilon_n}^{+}(tq_n)\what{\mu}(-tq_n), \\[2ex]
S_2(m,n) &:=&  \sum_{s\in \Z \setminus \{0\}}   \what{W}_{q_n,\gamma,\varepsilon_n}^+(0) \what{W}_{q_m,\gamma,\varepsilon_m}^{+}(sq_m)\what{\mu}(-sq_m), \\[2ex]
S_3(m,n) &:=& \mathop{ \mathop{\sum\sum}_{ s,t\in\mathbb{Z}\setminus \{0\} } }_{sq_m+tq_n\neq 0}\what{W}_{q_m,\gamma,\varepsilon_m}^{+}(sq_m)\what{W}_{q_n,\gamma,\varepsilon_n}^{+}(tq_n)\what{\mu}\left(-(sq_m+tq_n)\right).\vspace{2mm}
\end{eqnarray*}
Inequalities ~\eqref{fcoef_zero}, ~\eqref{ub_single_sum} and balance condition ~\eqref{growcond} imply that, for any $m,n \in \N,$ $m<n$,
\begin{equation}\label{ests1}
|S_1(m,n)| \ \ll \  \frac{(2+\varepsilon_m)\psi(q_m)}{n^{\rho}}\sum_{t\in \Z }  \what{W}_{q_n,\gamma,\varepsilon_n}^{+}(tq_n)  \ \ll \
\frac{\psi(q_m)}{n^{\rho}\varepsilon_n^{1/2}}\, .
\end{equation}
Symmetrically, for any $m,n \in \N,$ $m<n,$
\begin{equation}\label{ests2}
|S_2(m,n)| \ \ll \
\frac{\psi(q_n)}{m^{\rho}\varepsilon_m^{1/2}}\, .
\end{equation}

Now we decompose $S_3$ into two sums:
$$   S_3(m,n) \, =\, S_4(m,n) + S_5(m,n), $$
where
$$ S_4(m,n) \ := \mathop{ \mathop{\sum\sum}_{ s,t\in\mathbb{Z}\setminus \{0\} } }_{|sq_m-tq_n|\geq q_m^{\alpha}} \! \what{W}_{q_m,\gamma,\varepsilon_m}^{+}(-sq_m)\what{W}_{q_n,\gamma,\varepsilon_n}^{+}(tq_n)   \what{\mu}\left( sq_m-tq_n\right)  $$
and
\begin{equation} \label{def_S_7}
S_5(m,n) \ := \mathop{ \mathop{\sum\sum}_{ s,t\in\mathbb{Z}\setminus \{0\} } }_{1\leq |sq_m-tq_n|<q_m^{\alpha}}   \!\!\!\! \what{W}_{q_m,\gamma,\varepsilon_m}^{+}(-sq_m)\what{W}_{q_n,\gamma,\varepsilon_n}^{+}(tq_n)   \what{\mu}\left( sq_m-tq_n\right),
\end{equation}
here $\alpha \in (0,1)$ is the constant from $\alpha$-separability condition of Theorem \ref{mainTHM}.
Regarding $S_4$, by making use of balance condition \eqref{growcond} with the restriction $|sq_m-tq_n|\geq q_m^{\alpha}$ and inequality \eqref{ub_double_sum}, it follows that
\begin{equation}\label{ests4}
|S_4(m,n)| \ \ll \ \frac{1}{m^{\rho}}\mathop{\sum\sum}_{ s,t\in\mathbb{Z}\setminus \{0\} }
\left|\what{W}_{q_m,\gamma,\varepsilon_m}^{+}(-sq_m)\right| \left|\what{W}_{q_n,\gamma,\varepsilon_n}^{+}(tq_n)\right| \ \ll \
\frac{1}{m^{\rho}\varepsilon_m^{1/2}\varepsilon_n^{1/2}}\, .
\end{equation}

Finally, we decompose $S_5$ into two sums:
$$   S_5(m,n) \, =\, S_6(m,n) + E(m,n), $$
where
$$ S_6(m,n):= \mathop{\mathop{ \mathop{\sum\sum}_{s,t\in\mathbb{Z}\setminus \{0\} } }_{s> m^3/\psi(q_m)}}_{1\leq |sq_m-tq_n|<q_m^{\alpha}}\hspace{-3mm}   \what{W}_{q_m,\gamma,\varepsilon_m}^{+}(-sq_m)\what{W}_{q_n,\gamma,\varepsilon_n}^{+}(tq_n)  \what{\mu}\left( sq_m-tq_n\right) \,  \vspace{2mm}
$$
and
$$ E(m,n):= \mathop{\mathop{ \mathop{\sum\sum}_{s,t\in\mathbb{Z}\setminus \{0\} } }_{1\leq s \leq m^3/\psi(q_m)}}_{1\leq |sq_m-tq_n|<q_m^{\alpha}}\hspace{-3mm}  \what{W}_{q_m,\gamma,\varepsilon_m}^{+}(-sq_m)\what{W}_{q_n,\gamma,\varepsilon_n}^{+}(tq_n) \what{\mu}\left( sq_m-tq_n\right).
$$
The restriction $1 \leq |sq_m-tq_n| < q_m^{\alpha}$ implies that
\begin{equation} \label{index_restriction_first_boundsv}
0<\left|s-t\frac{q_n}{q_m}\right| < 1 \, .
\end{equation}
Hence,  if nonzero $s$ and $t$ satisfy \eqref{index_restriction_first_boundsv} then both necessarily  must have the same sign and also for each fixed integer $s$ there exists a set $T_s$ of at most two non-zero integers  $t$  satisfying  the restriction $1 \leq |sq_m-tq_n| < q_m^{\alpha}.$ So, we can write $S_6$ as a single sum
$$S_6(m,n)= \mathop{\mathop{\sum}_{ s> m^3/\psi(q_m) :  } }_{t \in T_s}
\what{W}_{q_m,\gamma,\varepsilon_m}^{+}(-sq_m) \what{W}_{q_n,\gamma,\varepsilon_n}^{+}(tq_n)\what{\mu}\left( sq_m-tq_n\right).$$
So, using the trivial bound  $|\what{\mu}(t)|  \leq 1 $  together with  \eqref{W_ub_psi} to bound $|\what{W}_{q_n,\gamma,\varepsilon_n}^{+}(t q_n)|$ and \eqref{W_ub_1s2} to bound  $ |\what{W}_{q_m,\gamma,\varepsilon_n}^{+}(sq_m)| $, we obtain that for any integers $ 1 \leq m < n $
\begin{eqnarray}\label{ests6}
|S_6(m,n)| & \ll  &   \mathop{\mathop{\sum}_{ s> m^3/\psi(q_m) :  } }_{t \in T_s}
|\what{W}_{q_m,\gamma,\varepsilon_m}^{+}(sq_m)| \ | \what{W}_{q_n,\gamma,\varepsilon_n}^{+}(tq_n)| \nonumber
 \\[2ex]
& \ll  & \sum_{s> m^3/\psi(q_m)}\frac{1}{s^2\psi(q_m)  \, \varepsilon_m }\psi(q_n)  \ \ll  \  \frac{\psi(q_n)}{m^3 \, \varepsilon_m } \, .
\end{eqnarray}
Working with $E(m,n)$ note, that in view of \eqref{estimpsiqn} the condition $s \leq m^3/\psi(q_m)$ implies that $s \le m^5.$ This, together
with the fact that $(q_n)_{n\in\N}$ is $\alpha$-separated implies that $E(m,n)$ is empty sum. Thus,
\begin{equation}\label{estime}
  E(m,n)=0.
\end{equation}
So, from upper bounds \eqref{ests1}, \eqref{ests2}, \eqref{ests4}, \eqref{ests6} and equality \eqref{estime}, for any fixed $a,b \in \N$ with $a<b,$
and all natural numbers $m,n,$ with $a\leq m<n\leq b:$
\begin{equation}\label{estsmn}
  |S_{m,n}| \ll \frac{\psi(q_m)}{n^{\rho}\varepsilon_n^{1/2}}\, + \frac{\psi(q_n)}{m^{\rho}\varepsilon_m^{1/2}}\, + \frac{1}{m^{\rho}\varepsilon_m^{1/2}\varepsilon_n^{1/2}}\, + \frac{\psi(q_n)}{m^3 \, \varepsilon_m } \,.
\end{equation}
Now, given $a\in\N$, define, for all $n\in\N$, $n\geq a$,
\begin{equation} \label{def_epsilon_n_main}
\varepsilon_n:=\min\left(1,\, \left( \sum_{k=a}^{n}\psi(q_k) \right)^{-1}\right).
\end{equation}
Then, since $\psi(k) \le 1$ for all $k \in \N$,
\begin{equation} \label{epsilon_inverse_bound}
\varepsilon_n^{-1} \, \leq \, \max\left(1,n\right) \, \le \, n, \quad n \ge a.
\end{equation}
Therefore,
$$\sum_{n=a}^{\infty}\frac{1}{n^{\rho}\varepsilon_n^{1/2}}  \, \le \, \sum_{n=a}^{\infty}\frac{1}{n^{\rho-\frac{1}{2}}} \, < \, \infty  \, ,$$
(recall that $\rho > 2$), and so, it follows from ~\eqref{estsmn} that for any $a,b \in \N,$ $a<b,$
\begin{equation} \label{sumofsmn}
\mathop{\sum\sum}_{a\leq m<n\leq b}|S_{m,n}| \, \ll \, \sum_{n=a}^{b} \psi(q_n) + \mathop{\sum\sum}_{a\leq m<n\leq b}\frac{1}{n^{\rho}\varepsilon_m^{1/2}\varepsilon_n^{1/2}}  +  \mathop{\sum\sum}_{a\leq m<n\leq b}\frac{\psi(q_n)}{m^3 \, \varepsilon_m } \,.
\end{equation}
For the third sum in the right hand side of \eqref{sumofsmn}, from inequalities \eqref{epsilon_inverse_bound}, we have
\begin{equation}\label{thirdsumest}
  \mathop{\sum\sum}_{a\leq m<n\leq b}\frac{\psi(q_n)}{m^3 \, \varepsilon_m } \le \mathop{\sum\sum}_{a\leq m<n\leq b}\frac{\psi(q_n)}{m^2} \ll
  \sum_{n=a}^{b} \psi(q_n), \quad a,b \in \N, \, a<b.
\end{equation}
In order to estimate the second sum in the right hand side of \eqref{sumofsmn} let's consider two cases.

\noindent Case 1: $\sum\limits_{k=a}^{b}\psi(q_k) > 1$. Then, by \eqref{def_epsilon_n_main},
$$ \frac{1}{n^{\rho}\varepsilon_m^{1/2}\varepsilon_n^{1/2}} \leq \frac{1}{n^{\rho}}\sum_{k=a}^{b}\psi(q_k),$$
and so, since $\rho >2,$
\begin{equation} \label{sum_const_leq_linear_case1}
\mathop{\sum\sum}\limits_{a\leq m<n\leq b} \frac{1}{n^{\rho}\varepsilon_m^{1/2}\varepsilon_n^{1/2}} \leq  \left(\mathop{\sum\sum}\limits_{a\leq m<n\leq b} \frac{1}{n^{\rho}}\right) \cdot \left(\sum_{k=a}^{b}\psi(q_k)\right) \ll \sum_{k=a}^{b}\psi(q_n).
\end{equation}

\noindent Case 2: $\sum\limits_{k=a}^{b}\psi(q_k) \le 1$. It follows, by \eqref{def_epsilon_n_main}, that $\varepsilon_n=1$ for all $a\leq n\leq b$,\vspace{-2mm}
therefore
$$
\frac{1}{n^{\rho}\varepsilon_m^{1/2}\varepsilon_n^{1/2}} = \frac{1}{n^{\rho}}.
$$
By using Lemma~\ref{lemsizeEnEm} with $\tau = \frac{\rho}{2}$, we get
$$\frac{1}{n^{\rho}\varepsilon_m^{1/2}\varepsilon_n^{1/2}} = \frac{1}{n^{\rho}}\leq\frac{\psi(q_n)}{n^{\frac{\rho}{2}}}, \quad \ n \in \N, $$
 and so
\begin{equation} \label{sum_const_leq_linear_case2}
\mathop{\sum\sum}\limits_{a\leq m<n\leq b} \frac{1}{n^{\rho}\varepsilon_m^{1/2}\varepsilon_n^{1/2}} \leq \mathop{\sum\sum}\limits_{a\leq m<n\leq b} \frac{\psi(q_n)}{n^{\frac{\rho}{2}}} \leq \mathop{\sum\sum}\limits_{a\leq m<n\leq b} \frac{\psi(q_n)}{m^{\frac{\rho}{2}}}  \ll \sum_{n=a}^b\psi(q_n).
\end{equation}
Both cases give us the same upper bound. So, estimates \eqref{thirdsumest}, \eqref{sum_const_leq_linear_case1} and \eqref{sum_const_leq_linear_case2}, together with \eqref{sumofsmn}, imply that, for all $a,b \in \N$ with $a<b$,
$$
\mathop{\sum\sum}_{a\leq m<n\leq b}|S_{m,n}| \, \ll \, \sum_{n=a}^{b}\psi(q_n) \, .
$$
Therefore, by \eqref{intersbound},
\begin{equation} \label{sum_const_leq_linear}
\mathop{\sum\sum}\limits_{a\leq m<n\leq b}\mu(E_m\cap E_n) \leq  \mathop{\sum\sum}\limits_{a\leq m<n\leq b}\what{W}_{m,n}^{+}(0)+ O\left(\sum_{n=a}^{b}\psi(q_n) \right) .
\vspace{3ex}
\end{equation}
We now turn our attention to estimating the first  term on the right hand of \eqref{sum_const_leq_linear}. Recall that according to \eqref{main_q_Harman_ub2}
\begin{equation*}
 \what{W}_{m,n}^{+}(0)  \, \leq  \,  4(1+\varepsilon_m)(1+\varepsilon_n)\psi(q_m)\psi(q_n)  +  O\left( (q_m,q_n)\min\Big(\frac{\psi(q_m)}{q_m},\frac{\psi(q_n)}{q_n}\Big)  \right) \,.
\end{equation*}
Since $(\varepsilon_n)_{n \ge a}$ is decreasing and $\varepsilon_n \le 1, \quad n \ge a,$ for all $m,n \in \N$ with $a\le m<n \le b,$
\begin{equation} \label{sum_eps_psipsi_developed}
4(1+\varepsilon_m)(1+\varepsilon_n)\psi(q_m)\psi(q_n) \, \leq \, 4\psi(q_m)\psi(q_n) + 12 \varepsilon_m \psi(q_m) \psi(q_n) \, .
\end{equation}
Once again, we will consider two cases.

\noindent Case 1:  $\sum\limits_{k=a}^b\psi(q_k)\le 1$.
It follows, by \eqref{def_epsilon_n_main}, that $\varepsilon_n=1$ for all $a\leq n\leq b$,\vspace{-2mm}
therefore
\begin{equation} \label{ub_sum_eps_psipsi_simple}
\mathop{\sum\sum}\limits_{a\leq m<n\leq b} \varepsilon_m\psi(q_m)\psi(q_n)\leq  \mathop{\sum\sum}\limits_{a\leq m<n\leq b}\psi(q_m)\psi(q_n)<\left(\sum_{n=a}^b\psi(q_n)\right)^2<\sum_{n=a}^b\psi(q_n).
\end{equation}

\noindent Case 2: $\sum\limits_{k=a}^b\psi(q_k) > 1 $. Then, by \eqref{def_epsilon_n_main},
\begin{equation} \label{ub_sum_eps_psipsi}
\begin{aligned}
\mathop{\sum\sum}\limits_{a\leq m<n\leq b} &\varepsilon_m\psi(q_m)\psi(q_n)= \mathop{\sum\sum}\limits_{a\leq m<n\leq b} \psi(q_m)\psi(q_n)\ \min\left(1, \left(  \sum_{k=a}^{m}\psi(q_k)  \right)^{-1} \right) =\\[2ex]
&= \
\mathop{\sum\sum}\limits_{a\leq m<n\leq b}\frac{\psi(q_m)\psi(q_n)}{\max\left(1, \sum_{k=a}^{m}\psi(q_k) \right)} \leq \
\sum_{n=a}^{b}\psi(q_n)\sum_{m=a}^{b} \frac{\psi(q_m)}{\sum_{k=a}^{m}\psi(q_k)}  \, .
\end{aligned}
\end{equation}
From Lemma \ref{lemdivserie} we have that, for all $a,b \in \N$ with $a<b$,
$$\sum_{m=a}^{b} \frac{\psi(q_m)}{\sum_{k=a}^{m}\psi(q_k)} \le 1- \log(\psi(q_a)) + \log \left(\sum_{n=a}^b\psi(q_n) \right).$$
This together with~\eqref{ub_sum_eps_psipsi} implies that
\begin{equation} \label{ub_sum_eps_psipsi_2}
\begin{aligned}
\mathop{\sum\sum}\limits_{a\leq m<n\leq b} \varepsilon_m\psi(q_m)\psi(q_n) \ &\leq \ \left(\sum_{n=a}^b\psi(q_n) \right)\left( 1- \log(\psi(q_a)) + \log \left(\sum_{n=a}^b\psi(q_n) \right)\right)\\[2ex]
\ &\ll \ \sum_{n=a}^b\psi(q_n) \cdot \log \left( \sum_{n=a}^b\psi(q_n) \right).
\end{aligned}
\end{equation}
Combining the estimates ~\eqref{main_q_Harman_ub2}, ~\eqref{sum_eps_psipsi_developed}, ~\eqref{ub_sum_eps_psipsi_simple} and \eqref{ub_sum_eps_psipsi_2} we find that
\begin{multline} \label{ub_sum_eps_psipsi_combined}
\mathop{\sum\sum}\limits_{a\leq m<n\leq b}  W_{m,n}^{+}(0) \leq  2\left( \sum_{n=a}^{b}\psi(q_n)\right)^2 +  O\left(\left(\sum_{n=a}^b\psi(q_n) \right) \log^+ \left(\sum_{n=a}^b\psi(q_n) \right)\right) +\\[2ex] + \  O\left( \mathop{\sum\sum}_{a\leq m<n\leq b}(q_m,q_n)\min\Big(\frac{\psi(q_m)}{q_m},\frac{\psi(q_n)}{q_n}\Big) \right)   \, ,
\end{multline}
here we used the inequality
$$\left( \sum_{n=a}^{b}\psi(q_n)\right)^2 = \sum_{n=a}^{b}\psi^2(q_n) + 2 \mathop{\sum\sum}\limits_{a\leq m<n\leq b} \psi(q_m)\psi(q_n) \ge $$
$$ \ge 2 \mathop{\sum\sum}\limits_{a\leq m<n\leq b} \psi(q_m)\psi(q_n).$$
On combining \eqref{sum_const_leq_linear} and  \eqref{ub_sum_eps_psipsi_combined} we find that
\begin{multline*} 
2 \mathop{\sum\sum}_{a\leq m<n\leq b} \mu(E_m\cap E_n) \ \leq \  4\left( \sum_{n=a}^{b}\psi(q_n)\right)^2 + O\left(\mathop{\sum\sum}_{a\leq m<n\leq b} (q_m,q_n)\min\left(\frac{\psi(q_m)}{q_m},\frac{\psi(q_n)}{q_n}\right)\right) \\[2ex]
+ O\left(\left(\sum_{n=a}^{b}\psi(q_n)\right) \log^+\left(\sum_{n=a}^{b}\psi(q_n)\right)+\sum_{n=a}^{b}\psi(q_n)\right).
\end{multline*}
Now we use Lemma \ref{lemsizeEn} to complete the proof of Lemma \ref{lemsizeEnEm}.
\end{proof}

\subsection{Applying Lemma \ref{ebc} to prove Theorem \ref{mainTHM} }
In this section we will prove our main result, namely Theorem \ref{mainTHM}. We deduce it from
Lemma \ref{ebc} using estimates from Lemma \ref{lemsizeEn} and Lemma \ref{lemsizeEnEm}. In order to
do this we will show that all conditions of mentioned above lemmas are satisfied under conditions of
Theorem \ref{mainTHM}.


\begin{proof}[Proof of Theorem~\ref{mainTHM}]
Note that under balance condition~\eqref{growcond_f} and~\eqref{growcond} condition~\eqref{cond2_2} (and, actually, condition \eqref{cond2} as well) is satisfied. Therefore, the convergence part of Theorem~\ref{mainTHM} follows from~\cite[Theorem~2]{povezazo} 
Because of this, we need to prove only the divergence part. So in the rest of the proof we assume
\begin{equation}\label{divercond}
  \Psi(N):=\sum_{n=1}^N\psi(q_n) \to \infty, \quad \text{ when }  N \to \infty.
\end{equation}

Note that all conditions of Lemma \ref{lemsizeEn} are satisfied under assumptions of Theorem \ref{mainTHM} and condition \eqref{divercond}.

We proceed by showing that it is enough to prove Theorem~\ref{mainTHM} with the extra assumption of~\eqref{estimpsiqn}. 
To this end, introduce two new auxiliary functions $\omega, \psi^{\ast}: \ca{A}\to [0,1]$,
$$\omega(q_n)=n^{-\tau}, \quad n \in \N,$$
$$\psi^{\ast}(q_n)= \max \{\psi(q_n), \omega(q_n)\}, \quad n \in \N.$$
Note that
$$\sum_{n=1}^{\infty}\omega(q_n) = \sum_{n=1}^{\infty} n^{-\tau} < \infty,$$
so by convergence part of Theorem~\ref{mainTHM} already justified in the beginning of the proof (or see ~\cite[Theorem~2]{povezazo}) we have that counting function
$R(x,N;\gamma,\omega, \ca{A})$ remains bounded as $N \to \infty,$ therefore, in view of Proposition~\ref{prop1}  (i),
$x \notin W_{\ca{A}}(\gamma; \omega).$ So, by Proposition~\ref{prop1}  (iii), we have
$$R(x,N;\gamma,\psi^{\ast}, \ca{A})=R(x,N;\gamma,\psi, \ca{A})+O(1).$$
This implies that the conclusion of the theorem~\eqref{countFSPresultsv} for $\psi^*$ is equivalent to~\eqref{countFSPresultsv} with original function $\psi$.
In the meantime, $\psi^{*}$ obviously satisfies condition \eqref{estimpsiqn}. So, without loss of generality, we can assume that $\psi$ satisfies condition \eqref{estimpsiqn}.

So, we have checked that under conditions of Theorem~\ref{mainTHM} and assumption \eqref{divercond} all conditions of Lemma \ref{lemsizeEn} and Lemma \ref{lemsizeEnEm} are fulfilled, and now we can start to apply Lemma~\ref{ebc}. Using \eqref{lem_sum_mu_good_result} and~\eqref{overlapfinal_finite_set_of_primesSV} on the right-hand side of~\eqref{integral_square_estimate},  we find that, for $f_n(x)$ and $f_n$ defined by~\eqref{Harman_choice_parameters} and for any $a,b\in\N$  with $a< b$,
\begin{equation}
\begin{aligned}\label{integral_square_main_O}
\int_{\I} \left(\sum_{n=a}^{b} \big( f_n(x) -  f_n \big) \right)^2\mathrm{d}\mu(x)\ \, =  \, O&\Bigg( \ \Big(\sum_{n=a}^b\psi(q_n)\Big) \Big(\log^+\Big(\sum_{n=a}^b\psi(q_n)\Big)+1\Big)
\\[1ex]
&+  \ \mathop{\sum\sum}_{a\leq m<n\leq b} (q_m,q_n)\min\left(\frac{\psi(q_m)}{q_m},\frac{\psi(q_n)}{q_n}\right)  \ \Bigg).
\end{aligned}
\end{equation}
Let $m \in \N $ be the smallest integer satisfying  $a\leq m\leq b$ such that
\begin{equation} \label{ie1_c}
\sum_{n=a}^{m}\psi(q_n)  \, \geq   \, \frac{1}{2}\sum_{n=a}^{b}\psi(q_n).
\end{equation}
Note that by the definition of $m$, we have that
\begin{equation} \label{ie2_c}
\sum_{n=m}^{b}\psi(q_n)  \, \geq  \, \frac{1}{2}\sum_{n=a}^{b}\psi(q_n)
\end{equation}
and that for any integer $n$ such that $m\leq n\leq b$
\begin{equation} \label{thm1_Psi_is_large}
2\Psi(n) \, = 2 \sum_{k=1}^{n} \psi(q_k) \, \geq  \,  \sum_{k=a}^b\psi(q_k).
\end{equation}
From inequalities ~\eqref{ie1_c}, ~\eqref{ie2_c} and \eqref{thm1_Psi_is_large} we have
$$
\left(\sum_{n=a}^b\psi(q_n)\right)\left(\log^+\left(\sum_{n=a}^b\psi(q_n)\right)+1\right) \leq  \
2\left(\sum_{n=m}^b\psi(q_n)\right)\left(\log^+\left(\sum_{n=a}^b\psi(q_n)\right)+1\right) \leq$$

$$\leq \
2\left(\sum_{n=m}^b\psi(q_n)\left(\log^+\left(2\Psi(n)\right)+1\right)\right)
\leq \
2\left(\sum_{n=m}^b\psi(q_n)\left(\log^+\Psi(n)+2\right)\right) \leq $$

$$\leq  \
2\left(\sum_{n=a}^b\psi(q_n)\left(\log^+\Psi(n)+2\right)\right)  \, .
$$
Therefore,
\begin{equation*}
\begin{aligned}
 \left(\sum_{n=a}^b\psi(q_n)\right)& \left(\log^+\left(\sum_{n=a}^b\psi(q_n)\right)
 +1\right) +
\\[1ex]
&+ \ \mathop{\sum\sum}_{a\leq m<n\leq b} (q_m,q_n)\min\left(\frac{\psi(q_m)}{q_m},\frac{\psi(q_n)}{q_n}\right) \ \leq \  2 \sum_{n=a}^b\phi_n  \, ,
\end{aligned}
\end{equation*}
\noindent where
\begin{equation} \label{def_varphi_main}
\phi_n:= \psi(q_n)\left(\log^+\Psi(n)+2\right)
+\sum_{m=1}^{n-1}(q_m,q_n)\min\left(\frac{\psi(q_m)}{q_m},\frac{\psi(q_n)}{q_n}\right) \, .
\end{equation}
This, together with \eqref{integral_square_main_O}, implies condition \eqref{ebc_condition1} of Lemma~\ref{ebc}.
Now we use Lemma~\ref{ebc} with $X, f_n(x)$ and
 $ f_n$ given by~\eqref{Harman_choice_parameters} and $\phi_n$ by  \eqref{def_varphi_main}. It is left to note that for
 any $n \in \N$, we have that $f_n \leq \phi_n$, $f_n \leq 2 $ and
 \begin{equation*}
\Phi(N):=\sum_{n=1}^N\phi_n  \, \leq  \,  \Psi(N)\left(\log^+\Psi(N)+2\right) \, + \, E(N) \, ,
\end{equation*}
where $ E(N) $ is given by \eqref{error2sv}. Counting statement \eqref{countFSPresultsv} now follows from \eqref{ebc_conclusion}.
\end{proof}

\vspace{13mm}

\vspace{7mm}

\noindent Volodymyr Pavlenkov: Department of Mathematics,
University of York,


\noindent\phantom{Volodymyr Pavlenkov: }Heslington, York, YO10
5DD, England;

\noindent\phantom{Volodymyr Pavlenkov: }Igor Sikorsky Kyiv Polytechnic Institute, Kyiv, Ukraine.

\noindent\phantom{Volodymyr Pavlenkov: }e-mail: pavlenkovvolodymyr@gmail.com

 \vspace{5mm}
 
 \noindent Evgeniy Zorin: Department of Mathematics,
University of York,


\noindent\phantom{Evgeniy Zorin: }Heslington, York, YO10
5DD, England.


\noindent\phantom{Evgeniy Zorin: }e-mail: evgeniy.zorin@york.ac.uk

\begin{thebibliography}{99}


\bibitem{MIMO} V. Beresnevich, A. Burr,B. Nazer, S. Velani, (eds): {\em Number Theory Meets Wireless Communications}, Mathematical Engineering. Springer (2020).

\bibitem{BerVelMTP}
V. Beresnevich, S. Velani: \emph{A mass transference principle and the Duffin-Schaeffer conjecture for
Hausdorff measures}, Ann. of Math. (2) 164 (2006), no. 3, 971–992.

 \bibitem{durham}
 V. Beresnevich, F. Ram\'{\i}rez,
S. Velani: {\em
Metric Diophantine approximation: aspects of recent work},
in Dynamics and Analytic Number Theory, Editors: Dmitry Badziahin, Alex Gorodnik, and Norbert Peyerimhoff. LMS Lecture Note Series 437, Cambridge University Press, (2016).   1--95.

\bibitem{Besik} A.S.~Besicovitch: {\em Sets of fractional dimensions (IV): on rational approximation to real
numbers}, J. Lond. Math. Soc. 9 (1934), 126-131.

\bibitem{blum} C. Bluhm: {\em Liouville numbers, Rajchman measures, and small Cantor sets}, Proceedings of the American Mathematical Society
Vol. 128, Numb. 9 (2000), 2637–2640.

\bibitem{Bl}
C. Bluhm: {\em Zur Konstruktion von Salem-Mengen}, PhD thesis, University of Erlanger, 1996.

\bibitem{B}
C. Bluhm: {\em On a theorem of Kaufman: Cantor-type construction of linear fractal Salem sets}, Ark. Mat. 36 (1998), no. 2, 307-316.

\bibitem{Bugeaund}
Y. Bugeaud: {\em Nombres de Liouville et nombres normaux}, C. R. Acad. Sci. Paris, Ser. I 335 (2002) 117–120.

\bibitem{Hambrook}
K. Hambrook: {\em Explicit Salem sets and applications to metrical Diophantine approximation}, 	arXiv:1604.00411 (2016)

 \bibitem{harman}
G. Harman: {\em Metric Number Theory.} Clarendon Press, Oxford (1998).

\bibitem{Jarnik} V. Jarnik, {\em Zur metrischen Theorie der diophantischen Approximation}, Prace Mat.-Fiz. 36
(1928-29), 91-106.

\bibitem{Kaufman}
R. Kaufman: {\em Continued Fractions and Fourier Transforms.}
Mathematika  27  (1980) 262--267.

 \bibitem{Khintchine24}
A. Khintchine: {\em Einige  S\"{a}tze \"{u}ber  Kettenbr\"{u}iche mit Anwendugen  auf die Theorie der diophantischen Approximationen.} Math. Ann. 92 (1924), 115-125.

\bibitem{IFS} J.~Li, T.~Sahlsten, {\em Trigonometric series and self-similar sets}, J. Eur. Math. Soc. 24(2022), 341–368.

\bibitem{Mattila1}
P. Mattila: \emph{Geometry of sets and measures in Euclidean spaces.} Cambridge Studies in Advanced
Mathematics, vol. 44, Cambridge University Press, 1995.

\bibitem{Mattila2}
P. Mattila: \emph{Fourier analysis and Hausdorff dimension}, Cambridge Studies in Advanced Mathematics, vol. 150, Cambridge University Press, 2015.


\bibitem{povezazo} A. Pollington, S. Velani, A. Zafeiropoulos, E. Zorin: {\em Inhomogeneous Diophantine Approximation on $M_0$-sets with
restricted denominators.} IMRN Vol. 2022, Issue 11 (2022), 8571–8643.

\bibitem{Rynne}
B. P. Rynne,\emph{ The Hausdorff dimension of sets arising from Diophantine approximation with a
general error function}, J. Number Theory 71 (1998), no. 2, 166–171.

\bibitem {QR}
M. Queff\'elec, O. Ramar\'e: {\em Analyse de Fourier des fractions continues \`a quotients restreints.} (French) [Fourier analysis of continued fractions with bounded special quotients] Enseign. Math. (2) 49 (2003), no. 3-4, 335-356.

\bibitem{schfracparts}
W. M. Schmidt: {\em Metrical theorems on fractional parts of sequences.}
Trans. Amer. Math. Soc. 110 (1964) 493-518.

 \bibitem{szusz}
P. Sz\"{u}sz:  {\em \"{U}ber die metrische Theorie der diophantischen Approximationen.} Acta Math. Acad. Sci. Hungar. 9 (1958), 177-193.

 \end{thebibliography}
\end{document}